\documentclass[11pt,twoside]{preprint}
\usepackage[a4paper,innermargin=1.2in,outermargin=1.2in,bottom=1.5in,marginparwidth=1.5in,marginparsep=3mm]{geometry}
\usepackage{amsmath,amsthm,amssymb,enumerate}
\usepackage{hyperref}
\usepackage{mathrsfs}
\usepackage{breakurl}
\usepackage{tikz}

\usepackage[osf]{newtxtext}
\usepackage[varqu,varl]{inconsolata}
\usepackage[cal=boondoxo]{mathalfa}
\usepackage[bigdelims,vvarbb,cmintegrals]{newtxmath}

\usepackage{xcolor}
\usepackage{mhequ}
\usepackage{comment}
\usepackage{microtype}

\colorlet{darkblue}{blue!90!black}
\colorlet{darkred}{red!90!black}

\makeatletter
\newenvironment{claim}[1][MM]
               {\list{$\bullet$}{%
  \setbox\@tempboxa\hbox{#1}\@tempdima\wd\@tempboxa%
  \setlength{\labelwidth}{\@tempdima}
  \advance\@tempdima by 1em%
  \setlength{\leftmargin}{\@tempdima}
  \setlength{\parsep}{1mm}\setlength{\itemindent}{0mm}%
  \setlength{\labelsep}{2mm}\setlength{\itemsep}{0mm}%
  \setlength{\topsep}{1mm}%
}}{\endlist}
\makeatother

\renewcommand{\theequation}{\thesection.\arabic{equation}}

\newtheorem{theorem}{Theorem}[section]
\newtheorem{lemma}[theorem]{Lemma}

\newtheorem{corollary}[theorem]{Corollary}

\theoremstyle{definition}
\newtheorem{assumption}[theorem]{Assumption}

\newtheorem{example}[theorem]{Example}

\theoremstyle{remark}
\newtheorem{remark}[theorem]{Remark}

\def\scal#1{\langle\, #1 \,\rangle}
\def\sscal#1{\langle\!\langle\, #1 \,\rangle\!\rangle}
\def\app#1{\stackrel{\text{(#1)}}{\approx}}

\newcommand{\vn}[1]{{\vert\kern-0.23ex\vert\kern-0.23ex\vert #1 
    \vert\kern-0.23ex\vert\kern-0.23ex\vert}}

\usetikzlibrary{snakes}
\usetikzlibrary{decorations}
\usetikzlibrary{positioning}
\usetikzlibrary{shapes}

\makeatletter
\pgfdeclareshape{crosscircle}
{
  \inheritsavedanchors[from=circle] 
  \inheritanchorborder[from=circle]
  \inheritanchor[from=circle]{north}
  \inheritanchor[from=circle]{north west}
  \inheritanchor[from=circle]{north east}
  \inheritanchor[from=circle]{center}
  \inheritanchor[from=circle]{west}
  \inheritanchor[from=circle]{east}
  \inheritanchor[from=circle]{mid}
  \inheritanchor[from=circle]{mid west}
  \inheritanchor[from=circle]{mid east}
  \inheritanchor[from=circle]{base}
  \inheritanchor[from=circle]{base west}
  \inheritanchor[from=circle]{base east}
  \inheritanchor[from=circle]{south}
  \inheritanchor[from=circle]{south west}
  \inheritanchor[from=circle]{south east}
  \inheritbackgroundpath[from=circle]
  \foregroundpath{
    \centerpoint%
    \pgf@xc=\pgf@x%
    \pgf@yc=\pgf@y%
    \pgfutil@tempdima=\radius%
    \pgfmathsetlength{\pgf@xb}{\pgfkeysvalueof{/pgf/outer xsep}}%
    \pgfmathsetlength{\pgf@yb}{\pgfkeysvalueof{/pgf/outer ysep}}%
    \ifdim\pgf@xb<\pgf@yb%
      \advance\pgfutil@tempdima by-\pgf@yb%
    \else%
      \advance\pgfutil@tempdima by-\pgf@xb%
    \fi%
    \pgfpathmoveto{\pgfpointadd{\pgfqpoint{\pgf@xc}{\pgf@yc}}{\pgfqpoint{-0.707107\pgfutil@tempdima}{-0.707107\pgfutil@tempdima}}}
    \pgfpathlineto{\pgfpointadd{\pgfqpoint{\pgf@xc}{\pgf@yc}}{\pgfqpoint{0.707107\pgfutil@tempdima}{0.707107\pgfutil@tempdima}}}
    \pgfpathmoveto{\pgfpointadd{\pgfqpoint{\pgf@xc}{\pgf@yc}}{\pgfqpoint{-0.707107\pgfutil@tempdima}{0.707107\pgfutil@tempdima}}}
    \pgfpathlineto{\pgfpointadd{\pgfqpoint{\pgf@xc}{\pgf@yc}}{\pgfqpoint{0.707107\pgfutil@tempdima}{-0.707107\pgfutil@tempdima}}}
  }
}
\makeatother

\colorlet{symbols}{black!50}
\definecolor{connection}{rgb}{0.7,0.1,0.1}
\definecolor{lblue}{rgb}{0.1,0.5,1}
\tikzset{
root/.style={circle,fill=black!50,inner sep=0pt, minimum size=3mm},
        dot/.style={circle,fill=black,inner sep=0pt, minimum size=1.5mm},
        empty/.style={circle,fill=white,inner sep=0pt, minimum size=1.5mm},
        bdot/.style={circle,fill=lblue,inner sep=0pt, minimum size=0.95mm},
        var/.style={circle,fill=black!10,draw=black,inner sep=0pt, minimum size=3mm},
        kernel/.style={semithick,shorten >=2pt,shorten <=2pt},
        kernels/.style={snake=zigzag,shorten >=2pt,shorten <=2pt,segment amplitude=1pt,segment length=4pt,line before snake=2pt,line after snake=5pt,},
        bkernel/.style={very thin,snake=zigzag,draw=lblue,shorten >=2pt,shorten <=2pt,segment amplitude=1.4pt,segment length=3.5pt,},
        rho/.style={densely dashed,semithick,shorten >=2pt,shorten <=2pt},
           testfcn/.style={dotted,semithick,shorten >=2pt,shorten <=2pt},
        renorm/.style={shape=circle,fill=white,inner sep=1pt},
        labl/.style={shape=rectangle,fill=white,inner sep=1pt},
        xic/.style={very thin,circle,fill=symbols,draw=black,inner sep=0pt,minimum size=1.2mm},
        Z/.style={very thin,rectangle,fill=blue!50,draw=black,inner sep=0pt,minimum size=1.2mm},
        Z2/.style={very thin,rectangle,fill=red!40,draw=black,inner sep=0pt,minimum size=1.2mm},
        xi/.style={very thin,circle,fill=blue!10,draw=black,inner sep=0pt,minimum size=1.1mm},
        arg/.style={very thin,rectangle,fill=blue!50,draw=black,inner sep=0pt,minimum size=1.2mm},
        xix/.style={crosscircle,fill=blue!10,draw=black,inner sep=0pt,minimum size=1.2mm},
	xib/.style={very thin,circle,fill=blue!10,draw=black,inner sep=0pt,minimum size=1.6mm},
	sqb/.style={very thin,rectangle,fill=blue!10,draw=black,densely dotted,inner sep=0pt,minimum size=1.6mm},
	xie/.style={very thin,circle,fill=green!50!black,draw=black,inner sep=0pt,minimum size=1.6mm},
	xid/.style={very thin,circle,fill=symbols,draw=black,inner sep=0pt,minimum size=1.6mm},
	xibx/.style={crosscircle,fill=blue!10,draw=black,inner sep=0pt,minimum size=1.6mm},
	kernels2/.style={very thick,draw=connection,segment length=12pt},
	kernels3/.style={very thick,draw=green!80!black,segment length=12pt},
	kernels4/.style={thick,draw=blue!70!black,segment length=12pt},
	dotkernel/.style={dashed,draw=black},
	not3/.style={thin,circle,draw=green!80!black,fill=green!80!black,inner sep=0pt,minimum size=0.5mm},
	not/.style={thin,circle,fill=symbols,draw=connection,fill=connection,inner sep=0pt,minimum size=0.5mm},
	>=stealth,
        }

\makeatletter
\def\DeclareSymbol#1#2#3{\expandafter\gdef\csname MH@symb@#1\endcsname{\tikz[baseline=#2,scale=0.15,draw=symbols,line join=round]{#3}}\expandafter\gdef\csname MH@symb@#1s\endcsname{\scalebox{0.7}{\tikz[baseline=#2,scale=0.15,draw=symbols,line join=round]{#3}}}}
\def\<#1>{\csname MH@symb@#1\endcsname}
\makeatother

\DeclareSymbol{0}{-2.8}{\node[xib] {};}

\DeclareSymbol{0x}{-2.8}{\node[xibx] {};}

\DeclareSymbol{1}{0}{\draw (0,0) node[not] {} -- (0,1.5) node[xi] {};}

\DeclareSymbol{r1}{0}{\draw[kernels2] (0,0) node[not] {} -- (0,1.5) node[xi] {};}

\DeclareSymbol{2}{-2}{\draw[kernels2] (-1,1) node[xi] {} -- (0,0) node[not] {} -- (1,1) node[xi] {};}

\DeclareSymbol{I2}{-3}{\draw[kernels2] (-1,1) node[xi] {} -- (0,0) node[not] {} -- (1,1) node[xi] {}; \draw (0,0) node[not] {} -- (0,-1) node[not] {};}

\DeclareSymbol{3}{-2}{\draw (0,-0.25) node[xi] {} -- (-1,1) node[xi] {};}

\DeclareSymbol{4}{-1}{\draw[kernels2] (-1,1) node[xi] {} -- (0,0) {}; \draw (0,0) node[not] {} -- (1,1) node[xi] {};}

\DeclareSymbol{5}{-3}{\draw[kernels2] (-1,1) node[xi] {} -- (0,0) node[not] {}; \draw (0,0) node[not] {} -- (0,-1.4) node[xi] {};}

\DeclareSymbol{6}{-3}{\draw (-1,1) node[xi] {} -- (0,0) {}; \draw[kernels2] (0,0) node[not] {} -- (0,-1.4) node[xi] {};}

\DeclareSymbol{7}{0.5}{\draw[kernels2] (-1,1)  -- (0,2) node[xi] {};
\draw[kernels2] (-1,1) node[not] {} -- (0,0) node[not] {} -- (1,1) node[xi] {};}

\DeclareSymbol{9}{0.5}{\draw[kernels2] (0,0)  -- (-1,1) node[xi] {};
\draw (0,1.4) node[xi] {} -- (0,0) node[not] {} -- (1,1) node[xi] {};}

\DeclareSymbol{3x}{-2}{\draw (0,-0.25) node[xix] {} -- (-1,1) node[xi] {};}

\DeclareSymbol{3x2}{-2}{\draw (0,-0.25) node[xi] {} -- (-1,1) node[xix] {};}

\DeclareSymbol{2x}{-2}{\draw[kernels2] (-1,1) node[xi] {} -- (0,0) node[not] {} -- (1,1) node[xi] {};
\draw (-1,0.95) node[not] {} -- (-1,1) node[xix] {};}

\DeclareSymbol{8}{-1}{\draw (-1,1) node[xi] {} -- (0,0) node[not] {} -- (1,1) node[xi] {};}

\DeclareSymbol{I2}{-4}{\draw[kernels2] (-1,1) node[xi] {} -- (0,0) node[not] {} -- (1,1) node[xi] {}; \draw (0,0) node[not] {} -- (0,-1.5) node[not] {}; }

\DeclareSymbol{I3}{-4}{\draw (0,-0.25) node[xi] {} -- (-1,1) node[xi] {};\draw (0,-0.25) node[xi] {} -- (0,-1.5) node[not] {};}

\DeclareSymbol{rI2}{-4}{\draw[kernels2] (-1,1) node[xi] {} -- (0,0) node[not] {} -- (1,1) node[xi] {}; \draw[kernels2] (0,0) node[not] {} -- (0,-1.5) node[not] {}; }

\DeclareSymbol{rI3}{-4}{\draw (0,-0.25) node[xi] {} -- (-1,1) node[xi] {};\draw[kernels2] (0,-0.25) node[xi] {} -- (0,-1.5) node[not] {};}

\DeclareSymbol{tmp}{-1}{\draw (-1,1) node[xi] {} -- (0,0) node[xi] {} -- (1,1) node[xi] {};}

\DeclareSymbol{II}{0.5}{\draw (-1,1)  -- (0,2) node[xi] {};
\draw (-1,1) node[xi] {} -- (0,0) node[xi] {};}

\DeclareSymbol{Z}{-2.8}{\node[Z] {};}

\DeclareSymbol{Z2}{-2.8}{\node[Z2] {};}

\DeclareSymbol{black}{-2.8}{\draw (-1,0) {} -- (1,0)}
\DeclareSymbol{blue}{-2.8}{\draw[kernels4] (-1,0) {} -- (1,0)}
\DeclareSymbol{red}{-2.8}{\draw[kernels2] (-1,0) {} -- (1,0)}
\DeclareSymbol{green}{-2.8}{\draw[kernels3] (-1,0) {} -- (1,0)}

\DeclareSymbol{ZI}{-0.5}{\draw (0,1.5) node[Z] {} -- (0,0) node[not]{};}

\DeclareSymbol{Z2I}{-0.5}{\draw (0,1.5) node[Z2] {} -- (0,0) node[not]{};}

\DeclareSymbol{ZJ}{-0.5}{\draw[kernels2] (0,1.5) node[Z] {} -- (0,0) node[not]{};}

\DeclareSymbol{Z2J}{-0.5}{\draw[kernels2] (0,1.5) node[Z2] {} -- (0,0) node[not]{};}

\DeclareSymbol{ZZ2}{-0.5}{\draw[kernels2] (-1,1) node[Z] {} -- (0,0) node[not]{} -- (1,1) node[Z2]{};}

\DeclareSymbol{ZIZ2I}{-0.5}{\draw (-1,1) node[Z] {} -- (0,0) node[not]{} -- (1,1) node[Z2]{};}

\DeclareSymbol{ZIZ2J}{-0.5}{\draw (-1,1) node[Z] {} -- (0,0) node[not]{}; \draw[kernels2] (0,0) node[not]{} -- (1,1) node[Z2]{};}

\DeclareSymbol{ZA}{-0.5}{\draw (-1,1) node[Z] {} -- (0,0) node[xi]{};}

\DeclareSymbol{ZM}{-0.5}{\draw (-1,1) node[xi] {} -- (0,0) node[Z]{};}

\DeclareSymbol{ZB}{-0.5}{\draw[kernels2] (-1,1) node[Z] {} -- (0,0) node[not]{} -- (1,1) node[xi]{};}

\DeclareSymbol{ZAM}{-0.5}{\draw (-1,1) node[Z] {} -- (0,0) node[xi]{} -- (1,1) node[xi]{};}

\DeclareSymbol{ZAA}{-0.5}{\draw (-1,1) node[Z] {} -- (0,0) node[xi]{} -- (-1,-1) node[xi]{};}

\DeclareSymbol{ZAB}{-0.5}{\draw (1,1) node[Z] {} -- (0,0) node[xi]{};
\draw[kernels2] (0,0) node[xi]{} -- (1,-1) node[not]{} -- (2,0) node[xi]{};}

\DeclareSymbol{ZBA}{-0.5}{\draw[kernels2] (-1,1) node[Z] {} -- (0,0) node[not]{} -- (1,1) node[xi]{};
\draw (0,0) node[not]{} -- (0,-1.4) node[xi]{};}

\DeclareSymbol{ZBB}{-0.5}{\draw[kernels2] (-1,1) node[Z] {} -- (0,0) node[not]{} -- (1,1) node[xi]{};
\draw[kernels2] (0,0) node[not]{} -- (1,-1) node[not]{} -- (2,0) node[xi]{};}

\DeclareSymbol{ZIM}{-0.5}{\draw (-1,1) node[Z] {} -- (0,0) node[not]{} -- (1,1) node[xi]{};}

\DeclareSymbol{0IM}{-0.5}{\draw (-1,1) node[xi] {} -- (0,0) node[not]{} -- (1,1) node[xi]{};}

\DeclareSymbol{0JM}{-0.5}{\draw[kernels2] (1,1) node[xi] {} -- (0,0) node[not]{};\draw (0,0) node[not]{} -- (-1,1) node[xi]{};}

\DeclareSymbol{ZBM}{-0.5}{\draw (0,1.5) node[xi] {} -- (0,0) {};
\draw[kernels2] (-1,1) node[Z] {} -- (0,0) node[not]{} -- (1,1) node[xi]{};}

\DeclareSymbol{SZ}{-0.5}{\draw (0,1.5) node[Z] {} -- (0,0) {};
\draw[kernels2] (-1,1) node[xi] {} -- (0,0) node[not]{} -- (1,1) node[xi]{};}

\DeclareSymbol{ZJM}{-0.5}{\draw (-1,1)[kernels2] node[Z] {} -- (0,0) {}; \draw (0,0) node[not]{} -- (1,1) node[xi]{};}

\DeclareSymbol{ZIJM}{-0.5}{\draw (1,1) node[Z] {} -- (0,0) {}; \draw[kernels2] (0,0) node[not]{} -- (-1,1) node[xi]{};}

\DeclareSymbol{AAA}{-0.5}{\draw (0,0) node[xi] {} -- (-1,1) node[xi] {} -- (0,2) node[xi] {};
\draw (0,0) node[xi] {} -- (-1,-1) node[xi] {};}

\DeclareSymbol{AAB}{-0.5}{\draw (0,0) node[xi] {} -- (-1,1) node[xi] {} -- (0,2) node[xi] {};
\draw[kernels2] (0,0) node[xi] {} -- (-1,-1) node[not] {} -- (-2,0) node[xi] {};}

\DeclareSymbol{AAM}{0.5}{\draw (-1,1)  -- (0,2) node[xi] {};
\draw (-1,1) node[xi] {} -- (0,0) node[xi] {} -- (1,1) node[xi] {};}

\DeclareSymbol{ABA}{0.5}{\draw (-1,1) -- (0,2) node[xi] {};
\draw[kernels2] (-1,1) node[xi] {} -- (0,0) node[not] {} -- (1,1) node[xi] {};
\draw (0,0) node[not] {} -- (-1,-1) node[xi] {};}

\DeclareSymbol{ABB}{0.5}{\draw (-1,1) -- (0,2) node[xi] {};
\draw[kernels2] (-1,1) node[xi] {} -- (0,0) node[not] {} -- (1,1) node[xi] {};
\draw[kernels2] (0,0) {} -- (-1,-1) node[not] {} -- (-2,0) node[xi] {};}

\DeclareSymbol{ABM}{0}{\draw (0,1.8) node[xi] {} -- (0,0.4) {};
\draw[kernels2] (-1,0) node[xi] {} -- (0,-1) node[not] {} -- (0,0.4) node[xi] {};
\draw (0,-1) node[not] {} -- (1,0) node[xi] {};}

\DeclareSymbol{AMA}{-4}{\draw (-1,1) node[xi] {} -- (0,0) node[xi] {} -- (1,1) node[xi] {};
\draw (0,0) node[xi] {} -- (0,-1.4) node[xi] {};}

\DeclareSymbol{AMB}{-4}{\draw (-1,1) node[xi] {} -- (0,0) node[xi] {} -- (1,1) node[xi] {};
\draw[kernels2] (0,0) node[xi] {} -- (0,-1.5) node[not] {} -- (1,-0.4) node[xi] {};}

\DeclareSymbol{AMM}{0.5}{\draw (0,0)  -- (0,1.4) node[xi] {};
\draw (-1,1) node[xi] {} -- (0,0) node[xi] {} -- (1,1) node[xi] {};}

\DeclareSymbol{BAA}{-4}{\draw[kernels2] (-1,1) node[xi] {} -- (0,0) node {} -- (1,1) node[xi] {};
\draw (0,0) node[not] {} -- (1,-1) node[xi] {} -- (0,-2) node[xi] {};}

\DeclareSymbol{BAB}{-4}{\draw[kernels2] (-1,1) node[xi] {} -- (0,0) node {} -- (1,1) node[xi] {};
\draw (0,0) node[not] {} -- (1,-1) node[xi] {};
\draw[kernels2] (1,-1) node[xi] {} -- (0,-2) node[not] {} -- (-1,-1) node[xi] {} ;}

\DeclareSymbol{BAM}{-4}{\draw[kernels2] (-1,1) node[xi] {} -- (0,0) node {} -- (1,1) node[xi] {};
\draw (0,0) node[not] {} -- (0,-1.5) node[xi] {} -- (1,-0.4) node[xi] {};}

\DeclareSymbol{BBA}{-4}{\draw[kernels2] (-1,1) node[xi] {} -- (0,0) node {} -- (1,1) node[xi] {};
\draw[kernels2] (0,0) node[not] {} -- (1,-1) {} -- (2,0) node[xi] {};
\draw (1,-1) node[not] {} -- (1,-2) node[xi] {};}

\DeclareSymbol{BBB}{-4}{\draw[kernels2] (-1,1) node[xi] {} -- (0,0) node {} -- (1,1) node[xi] {};
\draw[kernels2] (0,0) node[not] {} -- (1,-1) {} -- (2,0) node[xi] {};
\draw[kernels2] (1,-1) node[not] {} -- (2,-2) node[not] {} -- (3,-1) node[xi] {};}

\DeclareSymbol{BBM}{-4}{\draw[kernels2] (-1,1) node[xi] {} -- (0,0) node {} -- (1,1) node[xi] {};
\draw[kernels2] (0,0) node[not] {} -- (0,-1.4) {} -- (1,-0.4) node[xi] {};
\draw (0,-1.4) node[not] {} -- (-1,-0.4) node[xi] {};}

\DeclareSymbol{BMA}{-4}{\draw (0,0) -- (1,1) node[xi] {}; \draw[kernels2] (-1,1) node[xi] {} -- (0,0) node[not] {} -- (0,1.4) node[xi] {};
\draw (0,0) node[not] {} -- (0,-1.4) node[xi] {};}

\DeclareSymbol{BMB}{-4}{\draw (0,0) -- (1,1) node[xi] {}; \draw[kernels2] (-1,1) node[xi] {} -- (0,0) node[not] {} -- (0,1.4) node[xi] {};
\draw[kernels2] (0,0) -- (0,-1.5) node[not] {} -- (1,-0.4) node[xi] {};}

\DeclareSymbol{BMM}{0.5}{\draw (-1.2,0.6) node[xi]{}  -- (0,0) node[not] {} -- (1.2,0.6) node[xi]{};
\draw[kernels2] (-0.5,1.3) node[xi] {} -- (0,0) node[not] {} -- (0.5,1.3) node[xi] {};}

\DeclareSymbol{S1}{0}{\draw (-1,1.4) node[xi] {} -- (-1,0) {};
\draw (1,1.4) node[xi] {} -- (1,0) {};
\draw[kernels2] (-1,0) node[xi] {} -- (0,-1) node[not] {} -- (1,0) node[xi] {};}

\DeclareSymbol{ABMvar}{0}{\draw (1,1.4) node[xi] {} -- (1,0) {};
\draw[kernels2] (1,0) node[xi] {} -- (0,-1) node[not] {} -- (0,0.4) node[xi] {};
\draw (0,-1) node[not] {} -- (-1,0) node[xi] {};}

\DeclareSymbol{BBMvar}{0}{\draw[kernels2] (0.5,1.4) node[xi] {} -- (1,0) {} -- (1.5,1.4) node[xi] {};
\draw[kernels2] (1,0) node[not] {} -- (0,-1) node[not] {} -- (0,0.4) node[xi] {};
\draw (0,-1) node[not] {} -- (-1,0) node[xi] {};}

\DeclareSymbol{S2}{0}{\draw (-1,1.4) node[xi] {} -- (-1,0) {};
\draw[kernels2] (0.5,1.4) node[xi] {} -- (1,0) {} -- (1.5,1.4) node[xi] {};
\draw[kernels2] (-1,0) node[xi] {} -- (0,-1) node[not] {} -- (1,0) node[not] {};}

\DeclareSymbol{S2var}{0}{\draw (1,1.4) node[xi] {} -- (1,0) {};
\draw[kernels2] (-1.5,1.4) node[xi] {} -- (-1,0) {} -- (-0.5,1.4) node[xi] {};
\draw[kernels2] (-1,0) node[not] {} -- (0,-1) node[not] {} -- (1,0) node[xi] {};}

\DeclareSymbol{S3}{0}{\draw[kernels2] (-0.5,1.4) node[xi] {} -- (-1,0) {} -- (-1.5,1.4) node[xi] {};
\draw[kernels2] (1.5,1.4) node[xi] {} -- (1,0) {} -- (0.5,1.4) node[xi] {};
\draw[kernels2] (-1,0) node[not] {} -- (0,-1) node[not] {} -- (1,0) node[not] {};}

\DeclareSymbol{S4}{0}{\draw (0,1.8) node[xi] {} -- (0,0.4) {};
\draw[kernels2] (-1,0) node[xi] {} -- (0,-1) node[not] {} -- (1,0) node[xi] {};
\draw (0,-1) node[not] {} -- (0,0.4) node[xi] {};}

\DeclareSymbol{S5}{-4}{\draw[kernels2] (-1,1) node[xi] {} -- (0,0) node {} -- (1,1) node[xi] {};
\draw (0,0) node[not] {} -- (0,-1.4) {};
\draw[kernels2] (1,-0.4) node[xi] {} -- (0,-1.4) node[not] {} -- (-1,-0.4) node[xi] {};} 

\DeclareSymbol{rev1}{-4}{\draw[kernels4] (-1,1) node[xi] {} -- (0,0) node {};
\draw[kernels2] (0,0) {} -- (1,1) node[xi] {};
\draw[kernels3] (0,0) node[not3] {} -- (0,-1.4) node[not] {};
\draw[kernels2] (1,-0.4) node[xi] {} -- (0,-1.4) node[not] {} -- (-1,-0.4) node[xi] {} ;}

\DeclareSymbol{rev2}{-4}{\draw[kernels4] (-1,1) node[xi] {} -- (0,0) node {};
\draw[kernels2] (0,0) {} -- (1,1) node[xi] {};
\draw[kernels3] (0,0) node[not3] {} -- (0,-1.4) node[not] {};
\draw (1,-0.4) node[xi] {} -- (0,-1.4) node[xi] {} ;}

\DeclareSymbol{rev3}{0.5}{\draw (-1.2,0.6) node[xi]{}  -- (0,0) {};
\draw[kernels4] (0,0) {} -- (-0.5,1.3) node[xi]{};
\draw[kernels2] (0.5,1.3) node[xi] {} -- (0,0) node[not] {} -- (1.2,0.6) node[xi] {};}

\DeclareSymbol{rev4}{-1}{\draw[kernels4] (-1,1) node[xi] {} -- (0,0) {};
\draw[kernels2] (1,1) node[xi] {} -- (0,0) node[not] {};}

\newcommand{\bn}[1]{{[\kern-0.5ex] #1 
    [\kern-0.5ex]}}
\newcommand\subtree{\mathrel{\ooalign{$\subset$\cr
  \hidewidth\raise.1ex\hbox{$\cdot\mkern.5mu$}\cr}}}

\allowdisplaybreaks

\newcommand\Z{\mathbb{Z}}

\newcommand\bone{\mathbf{1}}

\newcommand\cR{\mathcal{R}}

\newcommand\cK{\mathcal{K}}

\newcommand\cB{\mathcal{B}}
\newcommand\cA{\mathcal{A}}
\newcommand\cC{\mathcal{C}}

\newcommand\cN{\mathcal{N}}

\newcommand\cI{\mathcal{I}}

\newcommand\cM{\mathcal{M}}
\newcommand\cW{\mathcal{W}}

\newcommand\cF{\mathcal{F}}

\newcommand{\cT}{\mathcal{T}}

\newcommand\scD{\mathscr{D}}

\newcommand\scT{\mathscr{T}}

\newcommand\frI{\mathfrak{I}}

\def\eps{\varepsilon}

\def\PPi{\mathbf{\Pi}}

\newcommand{\R}{{\mathbb{R}}}
\newcommand{\T}{\mathbb{T}}
\newcommand{\E}{\mathbf{E}}
\newcommand{\N}{\mathbf{N}}

\newcommand{\rem}{(\ldots)}
\newcommand{\rarr}{\rightarrow}
\newcommand{\darr}{\downarrow}

\def\span{\mathop{\mathrm{span}}}
\def\Sym{\mathop{\mathrm{Sym}}}
\def\d{\partial}
\def\id{\mathrm{id}}

\begin{document}
\title{Nondivergence form quasilinear heat equations driven by space-time white noise}
\author{M\'at\'e Gerencs\'er}
\institute{IST Austria}
\maketitle
\begin{abstract}
We give a Wong-Zakai type characterisation of the solutions of quasilinear heat equations driven by space-time white noise in $1+1$ dimensions.
In order to show that the renormalisation counterterms are local in the solution,
a careful arrangement of a few hundred terms is required.
The main tool in this computation is a general `integration by parts' formula that provides a number of linear identities for the renormalisation constants.
\end{abstract}
\tableofcontents

\section{Introduction}\label{sec:intro}
The main goal of the present paper is to `solve' the equation
\begin{equ}\label{eq:main}
\d_t u - a(u)\d_x^2 u=\xi
\end{equ}
on $\T=\R/\Z$, locally in time,
with some initial condition $u(0,\cdot)=u_0(\cdot)$, where $a:\R\to\R$ is a sufficiently regular function ($\cC^{5}$ suffices) with values in $[\lambda,\lambda^{-1}]$ for some $\lambda>0$,
and $\xi$ is the space-time white noise.

While equation \eqref{eq:main} looks like a simple nonlinear variation of the stochastic heat equation, a major problem arises due to the fact that
the product $a(u)\d_x^2u$ is not actually meaningful for $u$ with parabolic regularity less than $1$. 
Since the white noise $\xi$ has regularity less than $-3/2$, any reasonable solution of \eqref{eq:main} should have no more regularity than $1/2$, making the interpretation of the product on the left-hand side, and thus the equation, far from obvious.
One might try a na\"ive approximation: take a nonnegative symmetric (under the involution $x \mapsto -x$) 
smooth function $\rho$ supported in the unit ball and integrating to 1,
set $\rho^\eps(t,x)=\eps^{-3}\rho(\eps^{-2} t,\eps^{-1}x)$, $\xi^\eps=\rho^\eps\ast\xi$, and solve \eqref{eq:main} with $\xi^\eps$ in place of $\xi$.
While this sequence of solutions does not converge, one can `renormalise' the divergencies as follows.

\begin{theorem}\label{thm:main}
Let $u_0\in\cC^{2\delta}(\T)$ for some $\delta\in(0,1/2)$.
Then for any $\rho$ as above
there exist deterministic smooth functions $C^\eps_\cdot,\bar C^\eps_\cdot,\tilde C^\eps_\cdot$ such that the following holds.
Let $u^\eps$ be the classical solution of
\begin{equ}\label{eq:main approx}
\d_tu^\eps-a(u^\eps)\d_x^2u^\eps=\xi^\eps
+C^\eps_{a(u^\eps)}a'(u^\eps)
+\bar C^\eps_{a(u^\eps)}(a')^3(u^\eps)
+\tilde C^\eps_{a(u^\eps)}(a'a'')(u^\eps)
\end{equ}
on $\T$ with initial condition $u^\eps(0,\cdot)=u_0(\cdot)$. 
There exist some (random) $T>0$ and $u\in\cC^{\delta}([0,T]\times \T)$ that do not depend on $\rho$, such that $u^\eps\to u$ in probability in $\cC^{\delta}([0,T]\times \T)$.
\end{theorem}

In the case of semilinear SPDEs involving ill-defined products,
statements of the above kind on constructing renormalised solution theories have been plentiful in recent years, let us just mention the seminal works \cite{H0, GIP, Kup} from which most of them stem.
As for quasilinear equations,
slight variations of \eqref{eq:main} with noise regularity in $(-4/3,-1)$ were
considered around the same time in three different works \cite{OW, FGub, BDH}. The former was later extended to the regime $(-3/2,-1)$ in \cite{OWSS},
albeit only in the space-time periodic case. Removing the latter assumption in the regime $(-4/3,-1)$ or extending to more irregular noises (including space-time white noise as in our situation) is to our best knowledge work in progress \cite{Otto_Initial, OWSS2}.
We also remark that the divergence form version of \eqref{eq:main}, i.e. when $a(u)\d_x^2u$ is replaced by $\d_x(a(u)\d_x u)$, does not require the machinery of singular SPDEs, and has recently been treated in \cite{OWeb_Div2, OWeb_Div}.

A quite different approach was introduced in \cite{GH_Q}, which we will build on in the present article.
It relies on a transformation that brings \eqref{eq:main} to a form whose abstract counterpart in the language of regularity structures is relatively easily seen to be well-posed.
This argument is quite short and works for all range of noise regularity, and therefore provides a general solution theory.
In fact, the object $u$ from Theorem \ref{thm:main} that we will show to be the limit of $u^\eps$, is constructed in \cite{GH_Q}.
The drawback of this solution theory, however, is that it does not come with a natural approximation result, and therefore it is not a priori clear what, if anything, this abstract solution has to do with classical quasilinear PDEs.
Statements like Theorem \ref{thm:main} have the key role of relating the abstractly well-defined equation to classicaly well-defined equations.
It is actually natural to conjecture, but out of the scope of the current state of the theory, that this relation is `always' possible, as was proved in the semilinear case in \cite{BCCH}.

Let us now briefly outline what the source of difficulty is in obtaining such approximation results. 
To loosely recall the
transformation of \cite{GH_Q}
(its precise formulation is stated in Section \ref{sec:setup}),
the key observation is that quasilinear equations of the type \eqref{eq:main}
are (locally in time) equivalent to systems of the type
\begin{equ}
(u,v)=I\big(\hat F(\xi,u,v)\big),
\end{equ}
where $I$ is a convolution map satisfying certain Schauder estimates and $F$ is a subcritical nonlinearity.
In particular, $v$ is a nonlocal function of $u$.
This system can be also written abstractly within regularity structures:
\begin{equ}
(U,V)=\frI\big(\hat \cF(\Xi,U,V)\big),
\end{equ}
where the lift $\Xi$ of $\xi$ and the lift of $\hat \cF$ of $\hat F$ are as in \cite{H0},
and $\frI$ is the natural lift of $I$.
This already shows the first main issue: if one solves this equation with respect to a renormalised smooth model,
then the counterterms generated by the renormalisation will involve both $U$ and $V$.
Since in the renormalisation of the original equation one only expects to see local functions of the solution, we would then need that when reversing the transformation, the counterterms involving $V$ all magically disappear.

This is far from easy to verify: the number of these terms quickly blows up as the regularity of the noise decreases.
In the case of the space-time white noise, to calculate the counterterms at a single space-time point,
 the relevant dimension of the regularity structure is in the range of a few hundred.
It is worth noting that there is no symbolic cancellation between the terms that contribute to the renormalisation, and so the elimination of $V$ has to rely on cancellations between the renormalisation constants that different symbols generate.

This is our first main step: in Section \ref{sec:cancel}
we establish a number of symmetries that renormalisation constants satisfy.
This can be of interest on its own, for example one can
deduce the chain rule for the class of scalar-valued generalised KPZ equations from such cancellations, a question that goes back to \cite[Rem~1.14]{H0}.
Since such chain rule is part of a much more general study in the very recent work \cite{BGHZ}, we do not pursue this direction in any more detail here.
Armed with a sufficiently large class of cancellations, it then remains
to put them to use in simplifying the above mentioned large expression to the
form stated in Theorem \ref{thm:main}.
This is the main combinatorial task of the paper and is the content of Section \ref{sec:proof}.

Throughout the article we use concepts and terminology from the theory of regularity structures \cite{H0} without repeating any of the definitions, and to a low-level extent, from their renormalisation, see e.g. \cite[Sec~5]{H_Takagi} for a gentle introduction.

\subsection{Generalisations}
There are several directions for extensions of Theorem \ref{thm:main}. Some of them are immediate, some require mild improvement of known methods, and some would likely need new ideas.
\begin{claim}
\item The argument immediately extends to any Gaussian driving noise $\xi$ with regularity strictly above $-5/3$ and with compactly supported covariance function that satisfies the assumption of \cite[Sec~2.4]{CH}.
\item Instead of a spatially periodic setting, one can solve the equation with Dirichlet boundary conditions.
This direction for singular SPDEs was initiated in \cite{GH17}.
However, the application of its results is not completely automatic, as the construction of the extension $\hat\cR$ of the reconstruction operator $\cR$ below regularity $-1$ in highly nonlinear situation does require some work.
We believe that as long as one considers Dirichlet problems, this can be avoided, and everything above regularity $-2$ can be completely automatised.
A result of this flavor, but not of this generality, recently appeared in \cite[Sec~3]{Cyril_Ham}.
For Neumann boundary conditions such a statement is certainly not expected to hold.
In light of the results of \cite{GH17}, one in fact expects a boundary renormalisation to appear in the Neumann problem for \eqref{eq:main}.
\item For non-Gaussian noise, the regularity range $(-3/2,-1)$  would require a much simplified version of the computations in Section \ref{sec:proof}: instead of $17$ trees with $4$ noises, one needs to handle $6$ trees with $3$ noises.
When the regularity is between $-3/2$ and $-8/5$, one also gets an additional $6$ trees with $4$ noises, we briefly address this in Remark \ref{rem:Gauss}.
\item One could complicate the right-hand side to a general KPZ-like one, that is, to $f(u)(\d_x u)^2+g(u)\xi$.
Since our transformation already requires the `full' gKPZ regularity structure, this would not increase the number of trees.
However, the coefficient for each tree would get more complicated.
Carrying out the calculations of Section \ref{sec:cancel} in this generality `by hand' would require quite some additional effort. 
\item 
Both of the two latter generalisations (and even more the case of more irregular noise, where
the ad hoc computations would get humanly infeasible) point to the need 
of a systematised algebraic/combinatorial treatment, as has been developed in the semilinear case in \cite{BHZ, BCCH}.
One main difference to their setup is that our abstract integration operator $\frI$, while relatively easy to handle from the analytic point of view, makes the algebra more involved, see e.g \eqref{eq:u second order}.
\end{claim}

\textbf{Acknowledgements.}
MG was supported by the Austrian Science Fund (FWF) Lise Meitner programme M2250-N32.
Many thanks to the referees for several suggestions on improving the presentation of the paper.

\section{Integration by parts in renormalisation}\label{sec:cancel}
In this section we formulate some identities that
renormalisation constants arising from the renormalisation of regularity structures satisfy.
It is worth noting that 
here we do not use any Gaussianity assumption.
Concerning the main assumption below, Assumption \ref{as:reno} does restrict the generality compared to e.g. \cite{BHZ, CH} quite significantly, but
it allows us to work without the major algebraic complications therein,
and still obtain a number of cancellations that will suffice for the proof of Theorem \ref{thm:main}.

Certain symmetries were obtained
in the very recent work \cite{BGHZ} for multicomponent generalised KPZ equations
driven by space-time white noise.
Our approach here is different and the identities follow from relatively down-to-earth integration by parts-like arguments.
The formulation below furthermore fits well our purposes in Section \ref{sec:proof}, as it keeps track of which edges are and which are not required to have the same integration parameter (denoted by $c$ below) for the identities to hold.

\subsection{Formulation}
Take a regularity structure $\scT=(\cT,A,G)$ as in \cite{H0}.
We assume the notation
\begin{equ}
\cT=\bigoplus_{\alpha\in A} \cT_\alpha\;,\quad
\cT_\alpha=\overline{\span\{\tau_i:i\in I_\alpha\}},
\end{equ}
with some index sets $I_\alpha$, where
$\overline{\,\cdot\,}$ denotes the topological closure.
We denote $\hat\cW=\cup_{\alpha\in A}\{\tau_i:i\in I_\alpha\}$, $\hat\cW_-=\cup_{\alpha\in A\cap(-\infty,0)}\{\tau_i:i\in I_\alpha\}$,
by $\tilde \cW$ and $\tilde \cW_-$ subset of these sets containing $\tau_i$-s without any nonzero power of $X$
and by $\bar\cW$ and $\bar\cW_-$ the further subset of symbols with at least $2$ noise components.
\begin{remark}
Let us briefly comment on the different sets above.
The form of the vector space $\cT$ and its generator $\hat\cW$
is somewhat more involved than in the usual examples, for example the ones in \cite{H0}.
The reason for this generality is that it accommodates infinite dimensional regularity structures, which is required for quasilinear equations.
On the other hand, the renormalisation group in our setting will be sufficiently simple so that it is described by its action on $\bar\cW$.
Finally, $\tilde \cW$ can be viewed as the set possible subtrees of elements of $\bar\cW$.
\end{remark}
We assume that the scaling is parabolic
and that all $\tau\in\hat\cW$ satisfies $|\tau|>-2$.
We furthermore assume that $\scT$ is equipped with an integration operator
$\cI=\cI_c$ of order $2$ that corresponds to a kernel $K=K_c$ that is $2$-smoothing in the sense of \cite[As~5.1]{H0}, is supported in the unit ball, and satisfies
\begin{equ}
(\d_t-c\d_x^2)K_c=\delta_0+f_c,
\end{equ}
where $c>0$ is some constant and $f=f_c$ is a smooth function.
We also assume that $\scT$ is equipped with the abstract differentiation operator $\scD$
and we use the shorthand $\cI'=\scD\cI$.

We assume that elements of $\hat\cW$ are obtained after repeated uses of integration (possibly different ones from $\cI_c$) and multiplication operators
and therefore can be canonically represented by trees. We understand the notion of subtrees in the natural way.
If $\tau$ has $k$ subtrees isomorphic to $\bar \tau$, we denote by $\iota^i_\tau\bar \tau$, $i=1,\ldots,k$, all possible embeddings of $\bar\tau$ in $\tau$.
If $k>0$, we denote it by $\bar\tau\subset\tau$.
If $\sigma$ is a subtree of $\tau$, let $L_\sigma\tau$ be the tree obtained by contracting $\sigma$ to a node.
The action of these contractions on powers of $X$ appearing in the symbols will not play a role in our setting, for details on that we refer to \cite{H_Takagi} and for even more details to \cite{BHZ}.
For any map $g:\bar\cW_-\to\R$ we define $M_{[g]}:\cT\to\cT$ by
the linear and continuous extension of
\begin{equ}\label{eq:reno map}
\tau\mapsto M_{[g]}\tau:=\tau+\sum_{\bar\tau\in\bar\cW_- }g(\bar\tau)\sum_{i}L_{\iota^i_{\tau}\bar\tau}\tau,
\qquad \tau\in\hat\cW.
\end{equ}
Note that even in case $\bar\cW_-$ is infinite (which is the situation of Section \ref{sec:proof}), the sum in \eqref{eq:reno map} has finitely many nonzero contributions.

Fix a set of canonical models $\cM_0$ built from a class of approximate noises of a `target' noise $\xi$ (which may have multiple components). 
We will refer to elements of $\cM_0$ by $\PPi_\eps^\theta$, where
$\eps\in(0,1]$ and $\theta$ runs over some parameter set $\Theta$.
In the context of Theorem \ref{thm:main}, for example, $\Theta$ would be the set of all mollifiers $\rho$ of the form prescribed preceding the theorem.
As usual, we assume the translation invariance of the laws of the approximations, and we also assume that the $\sigma$-algebras
$\sigma\big((\PPi_\eps^\theta\tau_1)(z):\,z\in D_1\big)$ and
$\sigma\big((\PPi_\eps^\theta\tau_2)(z):\,z\in D_2\big)$ are independent if the distance between $D_1,D_2\subset\R^d$ is bigger than $R$, for some $R$ uniformly in $\eps,\theta, \tau_1,\tau_2$.
In a rather large generality \cite{CH} showed that 
one can find maps $\hat M_\eps^\theta:\cT\to\cT$ 
satisfying some natural conditions
 such that for all $\theta\in\Theta$ the models
$\hat\PPi_\eps^\theta:=\PPi_\eps^\theta \hat M_\eps^\theta$ converge in $L_p$ (in the probabilistic sense) to an admissible model $\hat\PPi$ as $\eps\to0$.
In a general situation these maps $\hat M_\eps^\theta$ and
what the `natural conditions' really mean
can be quite complicated,
here we restrict our attention to the following simplified case.
\begin{assumption}\label{as:reno}
The maps
$\hat M_\eps^\theta$ are of the form $M_{[\hat g_\eps^\theta]}$,
with 
\begin{equ}\label{eq:BPHZ g}
\hat g_\eps^\theta(\tau)=-(\E\PPi_\eps^\theta\tau)(0)+
\sum_{\tau\neq\bar\tau\in\bar\cW_- }\hat g_\eps^\theta(\bar\tau)
\sum_{i}(\E\PPi_\eps^\theta L_{\iota^i_{\tau}\bar\tau}\tau)(0).
\end{equ}
Moreover, for all $\tau\in\hat\cW$, $\bar\tau\in \bar\cW_-$, and embedding $\iota_\tau\bar\tau$,
one has 
$(\hat M_\eps^\theta-\id) L_{\iota_\tau\bar\tau}\tau=0$.
\end{assumption}
As for the notion of convergence, which is also somewhat involved, the only fact we will explicitly use
is that
for some $\alpha\in\R$ and all $\tau\in\hat\cW$, $\hat\PPi_\eps^\theta\tau$ converges to $\hat\PPi\tau$ in $L_p(\Omega,\cC^{\alpha}_{\mathop{\mathrm{loc}}})$.
It then follows that since $\E\hat\PPi\tau$ (as well as $\E (h\ast\hat\PPi\tau)\hat\PPi\bar\tau$ for any smooth function $h$) is a translation invariant distribution, it is actually a constant function, and its value depends only on the law of $\xi$.
Viewing \eqref{eq:BPHZ g} as a recursive definition of $\hat g_\eps^\theta$, it guarantees that
$(\E\hat\PPi_\eps^\theta\tau(0))=0$ for all $\tau\in\bar\cW_-$.
Assumption \ref{as:reno} also implies that for any $g$, $\hat\PPi_\eps^\theta M_{[g]}$ converges to $\hat\PPi M_{[g]}$, and the latter is also an admissible model.

\begin{remark}
Assumption \ref{as:reno} is discussed in the setting of \eqref{eq:main} in Section \ref{sec:setup}.
Let us also emphasise that Assumption \ref{as:reno} depends not only on $\scT$ but also on the choice of the approximations $\cM_0$.
It is general enough to cover for example symmetric (but not necessarily Gaussian) approximations of generalised KPZ equations.
It fails however, for example, for non-symmetric approximations of the KPZ equation:
when contracting in $\<BBB>$ the middle subtrees isomorphic to $\<7>$, one again gets $\<7>$, which in the non-symmetric case is not invariant under the renormalisation map.
\end{remark}

Let us extend $g$ as above as $0$ on $\hat\cW\setminus\bar\cW_-$.
With this convention,
denoting the set $\cN\subset\hat\cW$ such that for all $\tau\in\cN$ and all $\theta\in\Theta$ one has $\hat g^\theta_\eps(\tau)=0$,
$\cN$ always contains all symbols of positive degree.

The root of a tree $\tau$ is denoted by $\rho$ (with the understanding that it inherits the indices, so for example the root of a tree called $\tau_1$ will be denoted by $\rho_1$).
In the following $\tau_0$ always denotes a tree with a distinguished node (which may or may not be its root) $\rho_0^\ast$. By $\tau\circ\tau_0$ we denote the tree obtained from gluing $\tau$ and $\tau_0$ together by identifying $\rho$ and $\rho_0^\ast$.
In the special case $\rho_0=\rho_0^\ast$, one has simply $\tau\circ\tau_0=\tau\tau_0$.

Denote by $\bar\tau\subset_{\bullet}\tau$ if $\bar\tau$ can be embedded as a
subtree in $\tau$ that includes its root $\rho$. Given $\tau_0$,
denote by $\bar\tau\subset_{\ast}\tau_0$ if $\bar\tau$ can be embedded as a
subtree in $\tau_0$ that includes its distinguished node $\rho_0^\ast$. 
Summarising the possible inclusions in one example:
\begin{equ}
\tau_0=
\begin{tikzpicture}[scale=0.27,baseline={(0,-0.2)}]
{\draw[kernels2] (-1,1) node[xib] {} -- (0,0) node {} -- (1,1) node[xib] {};
\draw (0,0) node[not] {} -- (1,-1) node[xib] {} -- (0,-2) node[xib] {};
\draw (2.2,-1) node {\scriptsize $\rho_0^*$};}
\end{tikzpicture}
\qquad\text{e.g.: }\<2>\subset\tau_0,  \,\,  \<I2>\subset_\ast\tau_0,  \,\,  \<3>\subset_\bullet\tau_0.
\end{equ}
Introduce the following sets
\begin{equs}
\cA_1&=\{(\tau_1 & & ,\ldots,\tau_n):\,  n\geq 2,\;\tau_i\in\tilde\cW,
\quad
(\cI'\bar\tau_{i(1)}){\textstyle\prod_{k=2}^\ell}\cI\bar\tau_{i(k)}\in\cN
\label{eq:As}
\\
& & &
\quad\quad\forall \ell\in[2, n-1],\; i(1)\neq\cdots \neq i(\ell),\;
\bar\tau_{i(k)}\subset_{\bullet}\tau_{i(k)}
\},
\\
\cA_2&=\{(\tau_0 & & ,\ldots,\tau_n):\,  n\geq 2,\;\tau_i\in\tilde\cW,
(\tau_1,\ldots,\tau_n)\in\cA_1,
\quad
\cI\big((\cI'\bar\tau_{i(1)}){\textstyle\prod_{k=2}^\ell}\cI\bar\tau_{i(k)}\big)\circ\bar\tau_0\in\cN
\\
& & &\quad
\quad\forall \ell\in[1, n-1],\; i(1)\neq\cdots \neq i(\ell),\;
\bar \tau_0\subset_{\ast}\tau_0, \;
\bar\tau_{i(k)}\subset_{\bullet}\tau_{i(k)}\},
\\
\cA_3&=\{(\tau_0 & & ,\ldots,\tau_n):\,  n\geq 2,\;\tau_i\in\tilde\cW,
(\tau_1,\ldots,\tau_n)\in\cA_1,
\quad
\cI'\big((\cI'\bar\tau_{i(1)}){\textstyle\prod_{k=2}^\ell}\cI\bar\tau_{i(k)}\big)\circ\bar\tau_0\in\cN
\\
& & &\quad
\quad\forall \ell\in[1, n-1],\; i(1)\neq\cdots \neq i(\ell),\;
\bar \tau_0\subset_{\ast}\tau_0, \;
\bar\tau_{i(k)}\subset_{\bullet}\tau_{i(k)}
\}.
\end{equs}
Finally, if a real valued sequence $a_\eps$ converges to a finite limit depending only on the law of $\xi$, we denote it by $a_\eps\sim 0$.
Our `integration by parts' formulae then read as follows.
\begin{lemma}\label{lem:Symmetry}
Under Assumption \ref{as:reno}, one has for all $(\tau_1,\ldots,\tau_n)\in\cA_1$
\begin{equ}\label{eq:Symmetry}
\sum_{i=1}^n \hat g_\eps^\theta\big(\tau_i{\textstyle\prod_{k\neq i}\cI\tau_k}\big)
-
c\sum_{i=1}^n\sum_{i\neq j=1}^n
\hat g_\eps^\theta \big((\cI'\tau_i)(\cI'\tau_j)
\textstyle{\prod_{k\neq i,j}\cI\tau_k}\big)
\sim 0,
\end{equ}
for all $(\tau_0,\ldots,\tau_n)\in\cA_2$
\begin{equs}[eq:Symmetry2]
\sum_{i=1}^n\hat g_\eps^\theta\big(\cI\big(\tau_i &{\textstyle\prod_{k\neq i}}\cI\tau_k\big)\circ\tau_0\big)
-\hat g_\eps^\theta\big(({\textstyle\prod_{k}}\cI\tau_k)\circ\tau_0\big)
\\
&- c\sum_{i=1}^n\sum_{i\neq j=1}^n
\hat g_\eps^\theta\big(\cI\big((\cI'\tau_i)(\cI'\tau_j)
\textstyle{\prod_{k\neq i,j}\cI\tau_k}\big)\circ\tau_0\big)
\sim 0,
\end{equs}
and for all $(\tau_0,\ldots,\tau_n)\in\cA_3$
\begin{equs}[eq:Symmetry3]
\sum_{i=1}^n\hat g_\eps^\theta\big(\cI'\big(\tau_i &{\textstyle\prod_{k\neq i}}\cI\tau_k\big)\circ\tau_0\big)
-\sum_{i=1}^n\hat g_\eps^\theta\big(\big((\cI'\tau_i){\textstyle\prod_{k\neq i}}\cI\tau_k\big)\circ\tau_0\big)
\\
&- c\sum_{i=1}^n\sum_{i\neq j=1}^n
\hat g_\eps^\theta\big(\cI'\big((\cI'\tau_i)(\cI'\tau_j)
\textstyle{\prod_{k\neq i,j}\cI\tau_k}\big)\circ\tau_0\big)
\sim 0.
\end{equs}
\end{lemma}

The following corollary is immediate.
\begin{corollary}\label{cor:main}
Under Assumption \ref{as:reno} there exist maps $g_\eps^\theta:\bar\cW_-\to\R$  and such that
\begin{claim}
\item The identities \eqref{eq:Symmetry}-\eqref{eq:Symmetry2}-\eqref{eq:Symmetry3} are satisfied with equality;
\item The sequence of models $\PPi_\eps^\theta M_{[g_\eps^\theta]}$ converge
and the limit is of the form $\hat\PPi M_{[g]}$ for some $g$ depending only on the law of $\xi$;
\item If for $\tau_1,\ldots,\tau_k\in\bar\cW_-$ the system of equations
\begin{equ}
g_\eps^\theta(\tau_i)=0\quad\quad i=1,\ldots,k
\end{equ}
is linearly independent of \eqref{eq:Symmetry}-\eqref{eq:Symmetry2}-\eqref{eq:Symmetry3}, then $g_\eps^\theta$ can be chosen to agree with $\hat g_\eps^\theta$ on $\tau_1,\ldots,\tau_k$.
\end{claim}
\end{corollary}

\begin{remark}
One can pictorially represent the above as follows.
Focusing on the $n=2$ case, the identities \eqref{eq:Symmetry} give relationships between renormalisation constants of trees obtained from the `scheme'
\begin{tikzpicture}[scale=0.27,baseline={(0,0)}]
\draw[densely dotted] (-1,1) node[sqb] {} -- (0,0) node {} -- (1,1) node[sqb] {};
\end{tikzpicture},
where the different edges 
\begin{tikzpicture}[scale=0.27]
\draw[densely dotted] (0,0) node {} -- (1.5,0) node {};
\end{tikzpicture}
are substituted with different combinations of $\cI$, $\cI'$, or contracting the edge, and $\cA_1$ gives conditions on what trees can be substituted in the placeholders
\begin{tikzpicture}[scale=0.27]
\draw[densely dotted] (0,0) node[sqb] {};
\end{tikzpicture} .
Similarly, for \eqref{eq:Symmetry2} and \eqref{eq:Symmetry3} one substitutes in the `scheme' 
\begin{tikzpicture}[scale=0.27,baseline={(0,-0.1)}]
\draw[densely dotted] (-1,1) node[sqb] {} -- (0,0) node {} -- (1,1) node[sqb] {} -- (0,0) node {} -- (1,-1) node[sqb] {};
\end{tikzpicture} .
\end{remark}

\begin{example}\label{example}
Let us list a couple of examples in the case $c=1$.
We use the graphical convention (as in, for example, \cite{H_String}) of $\<0>$ denoting the noise, $\<black>$ denoting $\cI$ and $\<red>$ denoting $\cI'$.
Then, assuming $|\,\<0>\,|>-3/2-1/100$ and $\<7>\,,\,\<4>\,,\,\<5>\in\cN$,
one has:
\begin{equs}
g^\theta_\eps(\<AMM>)&=3g^\theta_\eps(\<BMM>)
&\quad & \eqref{eq:Symmetry},\; n=4, \tau_1=\tau_2=\tau_3=\tau_4=\<0>\;,\\
&=2g^\theta_\eps(\<AMB>)-g^\theta_\eps(\<AMA>)
&\quad & \eqref{eq:Symmetry},\; n=2, \tau_1=\<0>\;,\tau_2=\<tmp>\;,\\
&=4g^\theta_\eps(\<ABM>)+2g^\theta_\eps(\<S4>)-2g^\theta_\eps(\<AAM>)
&\quad & \eqref{eq:Symmetry},\; n=3, \tau_1=\tau_2=\<0>\;,\tau_3=\<3>\;,
\\
6g_\eps^\theta(\<BMA>)&=3g_\eps^\theta(\<AMA>)-g_\eps^\theta(\<AMM>)
&\quad & \eqref{eq:Symmetry2},\; n=3, \tau_0=\tau_1=\tau_2=\tau_3=\<0>\;,
\\
2g_\eps^\theta(\<BBB>)
&=g_\eps^\theta(\<BAB>)+g_\eps^\theta(\<BMB>)-g_\eps^\theta(\<BBM>)-g_\eps^\theta(\<S5>)
&\quad & \eqref{eq:Symmetry3},\; n=2, \tau_0=\<r1>\;, \tau_1=\<0>\;, \tau_2=\<3>\;,
\\
&=2g_\eps^\theta(\<ABB>)-2g_\eps^\theta(\<BMB>)
&\quad & \eqref{eq:Symmetry3},\; n=2, \tau_0=\<7>\;,\tau_1=\tau_2=\<0>\;.
\end{equs}
In the last example we chose $\rho_0^\ast\neq\rho_0$ to be the leftmost node in $\<7>$.

Let us check that the given tuples indeed lie in the appropriate $\cA_i$ sets, in the first and last example above. To verify $(\<0>\,,\<0>\,,\<0>\,,\<0>)\in\cA_1$, since there are no nontrivial subtrees of $\<0>$, the only choice in \eqref{eq:As} is $\ell$: for $\ell=2$ we have $\<4>\in\cN$ by assumption and for $\ell=3$ we have $\<9>\in\cN$ since $|\,\<9>\,|>0$. To verify $(\<7>\,,\<0>\,,\<0>)\in\cA_3$, first we note that $(\<0>\,,\<0>)\in\cA_1$ is automatic.
Since there are no nontrivial subtrees of $\tau_1=\tau_2=\<0>$ and the only choice for $\ell$ is $\ell=1$, the other condition in $\cA_3$ boils down to show $\tau=\cI'(\cI'\<0>)\circ\bar\tau_0\in\cN$ for $\bar\tau_0\subset_\ast\tau_0$. For the choice $\bar\tau_0=\<r1>$ this holds by assumption, while for all other choices $|\tau|>0$.
\end{example}

\subsection{Proof of Lemma \ref{lem:Symmetry}}
We will sometimes use $\d^{[z]}_\alpha$ to emphasise that a differential operator $\d_\alpha$ acts on the $z$ variable.
We say that a function $Q$ in $n$ $d$-dimensional variables is translation-invariant if
$T_{\bar z}Q(z_1,\ldots,z_n):=Q(z_1+\bar z,\ldots,z_n+\bar z)=Q(z_1,\ldots,z_n)$ for all $\bar z\in\R^d$.
Notice that if $Q$ is a translation-invariant smooth function and $\eta$ is a compactly supported distribution, then for any nonzero multiindex $\alpha$
\begin{equ}\label{eq:translation vanish}
\big(\d_\alpha^{[\bar z]}(T_{\bar z}\eta)\big)(Q)
=0.
\end{equ}
\begin{proof}[Proof of Lemma \ref{lem:Symmetry}]
To ease the notation, we drop the $\theta$ index, but it will be clear that
the conclusion is approximation-independent.
We start with the proof of \eqref{eq:Symmetry}.
Denote
\begin{equ}\label{eq:something}
\sigma={\textstyle\prod_{k}\cI\tau_k},\quad
\sigma^i=\tau_i{\textstyle\prod_{k\neq i}\cI\tau_k},\quad
\mathring\sigma^i={\textstyle\prod_{k\neq i}\cI\tau_k},\quad
\sigma_{ij}=(\cI'\tau_i)(\cI'\tau_j)\textstyle{\prod_{k\neq i,j}\cI\tau_k},
\end{equ}
as well as 
\begin{equ}
Q_\eps(z_1,\ldots,z_n)=\E\big(\PPi_\eps\tau_1(z_1)\cdots\PPi_\eps\tau_n(z_n)\big).
\end{equ}
Clearly $Q_\eps$ is a translation invariant smooth function which is $0$
on $\{\max_{i}\min_{j\neq i}|z_i-z_j|\geq R\}$.
Let us take a smooth compactly supported function $\chi^R$ that is $1$ on the ball of radius $R+2$ around the origin and denote $f^R=f\chi^R$.
From the Leibniz rule
\begin{equs}\label{eq:Leibniz}
{\textstyle\sum_i}
&(\delta_0+f)(\bar z-z_i){\textstyle\prod_{k\neq i}}K(\bar z-z_k)
-c
{\textstyle\sum_{i\neq j}}K'(\bar z-z_i)K'(\bar z-z_j)
{\textstyle\prod_{k\neq i,j}}K(\bar z-z_k)
\\
\quad &=
(\d_t-c\Delta)^{[\bar z]}\Big(K(z_0-\bar z){\textstyle\prod_{k}}K(\bar z-z_k)\Big)
\end{equs}
we get
\begin{equs}[eq:00]
{\textstyle\sum_i} & (\E\PPi_\eps\sigma^i)(\bar z)  -c{\textstyle\sum_{i\neq j}}(\E\PPi_\eps\sigma_{ij})(\bar z)
\\
&=\int  Q_\eps(z_1,\ldots,z_n)
(\d_t-c\Delta)^{[\bar z]}
\Big({\textstyle\prod_{k}}K(\bar z-z_k)\Big)\,dz_1\cdots dz_n
\\
&\quad - 
{\textstyle\sum_i}\int Q_\eps(z_1,\ldots,z_n)
f(\bar z-z_i)
{\textstyle\prod_{k\neq i}}K(\bar z-z_k)
\,dz_1\cdots dz_n
\\
&=-{\textstyle\sum_i}\big(\E (f^R\ast\PPi_\eps\tau_i)(\PPi_\eps\mathring\sigma^i\big))(\bar z),
\end{equs}
where we used that the integral in the second line vanishes due to \eqref{eq:translation vanish}.

There are two essentially different scenarios in which one has
$\bar\tau\subset\sigma^i.$ 
First, when $\iota_{\sigma^i}\bar\tau$ is obtained from an embedding $\iota_{\tau_\ell}\bar\tau$ for some
$\ell$.
This has obvious corresponding embeddings $\iota_{\sigma^j}\bar\tau$, $\iota_{\sigma_{ij}}\bar\tau$,
and $\iota_{\mathring\sigma^j}\bar\tau$, and moreover
the results of contracting these subtrees are exactly of the form \eqref{eq:something}, with $\tau_\ell$ replaced by $L_{\iota_{\tau_\ell}\bar\tau}\tau_\ell$.
In this case therefore one has an identity analogous to \eqref{eq:00}.

If $\bar\tau\subset\sigma^i$ is not of this form, then it can be written as
$\bar\tau=\tilde\sigma^i:=\tilde\tau_i{\textstyle\prod_{k}\cI\tilde\tau_{\ell(k)}}$
with some indices $\ell(k)$ distinct from each other and from $i$,
and with $\tilde\tau_j\subset_{\bullet}\tau_j$.
Denote the set of indices $\ell(k)$ along with $i$ by $I$.
One can pair these subtrees with those of $\sigma^j$ and $\sigma_{jm}$ whenever $j,m\in I$: simply define $\tilde\sigma^j\subset\sigma^j$ and
$\tilde\sigma_{ij}\subset\sigma_{ij}$ as in \eqref{eq:something}, replacing each $\tau_k$ with $\tilde\tau_k$.
One then has $L_{\tilde\sigma^j}\sigma^j=L_{\tilde\sigma^{j'}}\sigma^{j'}
=L_{\tilde\sigma_{mm'}}\sigma_{mm'}=:\hat\tau$ for all $j,j',m,m'\in I$.
Denote the set of all the possible $\hat\tau$-s obtained this way (with multiplicities) by $A$.
By Assumption \ref{as:reno}, for all $\hat\tau\in A$ one has $\PPi_\eps\hat\tau=\hat\PPi_\eps\hat\tau$.

The definition of $\cA_1$ guarantees that these two cases exhaust all the contributions to the renormalisation of
$\sigma_{ij}$ as well: the only subtrees not covered so far are of the form $\bar \tau=(\cI'\tilde\tau_i){\textstyle\prod_{k}\cI\tilde\tau_{\ell(k)}}$ with some indices $\ell(k)$ distinct from each other and from $i$ and $j$.
By definition, any $\bar\tau$ of this form belongs to $\cN$, so does not contribute to the renormalisation.

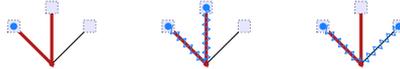
\begin{figure}[h]
\centering
\begin{tikzpicture}[scale=0.5]
\draw[kernels2] (-1,1) node[sqb] {} -- (0,0) node[not] {} -- (0,1.5) node[sqb] {};
\draw (0,0) node[not] {} -- (1,1) node[sqb] {};
\draw[bkernel] (-1,1) node[bdot] {};
\end{tikzpicture}
\qquad
\begin{tikzpicture}[scale=0.5]
\draw[kernels2] (-1,1) node[sqb] {} -- (0,0) node[not] {} -- (0,1.5) node[sqb] {};
\draw (0,0) node[not] {} -- (1,1) node[sqb] {};
\draw[bkernel] (-1,1) node[bdot] {} -- (0,0) node[not] {} -- (0,1.5) node[bdot] {};
\end{tikzpicture}
\qquad
\begin{tikzpicture}[scale=0.5]
\draw[kernels2] (-1,1) node[sqb] {} -- (0,0) node[not] {} -- (0,1.5) node[sqb] {};
\draw (0,0) node[not] {} -- (1,1) node[sqb] {};
\draw[bkernel] (-1,1) node[bdot] {} -- (0,0) node[not] {} -- (1,1) node[bdot] {};
\end{tikzpicture}
\caption{\footnotesize Different possible positions of subtrees, drawn by blue, for $\sigma_{12}$, $n=3$. The third kind belongs to $\cN$ by the definition of $\cA_1$.}
\end{figure}

Therefore we have
\begin{equs}
{\textstyle\sum_i} & (\E\hat\PPi_\eps\sigma^i)(\bar z)  -c{\textstyle\sum_{i\neq j}}(\E\hat\PPi_\eps\sigma_{ij})(\bar z)+{\textstyle\sum_i}\big(\E (f^R\ast\hat\PPi_\eps\tau_i)(\hat\PPi_\eps\mathring\sigma^i\big))(\bar z)
\\
&=\sum_{\hat\tau\in A} (\E\hat\PPi_\eps\hat\tau)(\bar z)\Big({\textstyle\sum_i} \hat g_\eps(\tilde\sigma^i)  -c{\textstyle\sum_{i\neq j}}\hat g_\eps(\tilde\sigma_{ij})\Big).
\end{equs}
After testing with any function integrating to $1$ and recalling that $\E\hat\PPi\tau$ is a constant distribution, we pass to the $\eps\to 0 $ limit and rearrange the above as
\begin{equs}
\lim_{\eps\to 0} & {\textstyle\sum_i} \hat g_\eps(\sigma^i)  -c{\textstyle\sum_{i\neq j}}\hat g_\eps(\sigma_{ij})=
\sum_{\bone\neq\hat\tau\in A} (\E\hat\PPi\hat\tau)(0)\lim_{\eps\to 0}\Big({\textstyle\sum_i} \hat g_\eps(\tilde\sigma^i)  -c{\textstyle\sum_{i\neq j}}\hat g_\eps(\tilde\sigma_{ij})\Big)
\\
&+{\textstyle\sum_i}  (\E\hat\PPi\sigma^i)(0)  -c{\textstyle\sum_{i\neq j}}(\E\hat\PPi\sigma_{ij})(0)+{\textstyle\sum_i}\big(\E (f^R\ast\hat\PPi\tau_i)(\hat\PPi\mathring\sigma^i\big))(0).
\end{equs}
Therefore by a simple induction argument we get the claim.

The proof of the other two claims goes along very similar lines.
There are two slight differences, the first one of which is the definition of $Q_\eps$: we now set
\begin{equ}
Q_\eps(z_0,z_1,\ldots,z_n)=\E\big(\PPi_\eps\tau_0^\ast(z_0)\PPi_\eps\tau_1(z_1)\cdots\PPi_\eps\tau_n(z_n)\big),
\end{equ}
where $\tau_0^\ast$ is the tree obtained by viewing $\tau_0$ as a tree with root $\rho_0^\ast$.
The other difference is the application of the Leibniz rule:
 we simply replace the identity \eqref{eq:Leibniz} with,
in the case of \eqref{eq:Symmetry2}
\begin{equs}[eq:01]
K(z_0-\bar z)
&{\textstyle\sum_i}
(\delta_0+f)(\bar z-z_i){\textstyle\prod_{k\neq i}}K(\bar z-z_k)
-(\delta_0+f)(\bar z-z_0){\textstyle\prod_{k}}K(\bar z-z_k)
\\
&\quad-cK(z_0-\bar z)
{\textstyle\sum_{i\neq j}}K'(\bar z-z_i)K'(\bar z-z_j)
{\textstyle\prod_{k\neq i,j}}K(\bar z-z_k)
\\
&=\d_t^{[\bar z]}\Big(K(z_0-\bar z){\textstyle\prod_{k}}K(\bar z-z_k)\Big)
\\
&\quad
-c\d_x^{[\bar z]}\Big(
K(z_0-\bar z)
{\textstyle\sum_i}
K'(\bar z-z_i){\textstyle\prod_{k\neq i}}K(\bar z-z_k)
+K'(z_0-\bar z){\textstyle\prod_{k}}K(\bar z-z_k)
\Big),
\end{equs}
while in the case of \eqref{eq:Symmetry3}
we make use of
\begin{equs}[eq:02]
K'(z_0-\bar z)&{\textstyle\sum_i}
(\delta_0+f)(\bar z-z_i){\textstyle\prod_{k\neq i}}K(\bar z-z_k)
-(\delta_0+f)(\bar z-z_0){\textstyle\sum_i}K'(\bar z-z_i){\textstyle\prod_{k\neq i}}K(\bar z-z_k)
\\
&\quad-cK'(z_0-\bar z)
{\textstyle\sum_{i\neq j}}K'(\bar z-z_i)K'(\bar z-z_j)
{\textstyle\prod_{k\neq i,j}}K(\bar z-z_k)
\\
&=\d_t^{[\bar z]}\Big(K'(z_0-\bar z){\textstyle\prod_{k}}K(\bar z-z_k)\Big)
-\d_x^{[\bar z]}\Big((\d_t K)(z_0-\bar z){\textstyle\prod_{k}}K(\bar z-z_k)\Big)
\\
&\quad-c\d_x^{[\bar z]}
\Big(K'(z_0-\bar z){{\textstyle\sum_{i}}K'(\bar z-z_i)\textstyle\prod_{k\neq i}}K(\bar z-z_k)\Big).
\end{equs}
The integral of $Q_\eps$ against the right-hand sides of \eqref{eq:01}
and \eqref{eq:02} vanishes as before due to \eqref{eq:translation vanish},
and hence the proof can be concluded precisely as before.

\end{proof}

\section{Proof of the main theorem}\label{sec:proof}

\subsection{The setup}\label{sec:setup}
We briefly recall the setup of \cite{GH_Q}.
For simplicity for certain `sufficiently large' indices from therein we simply take $10$, which suffices for
\eqref{eq:main}, but which does not play any important role.
The approach of \cite{GH_Q} relies on a transformation, which is of course formal for rough $\xi$, but is elementary to check for smooth $\xi$.
Let, for $c\in[\lambda,\lambda^{-1}]$, $P(c,\cdot)$ be the Green's function of the operator $\d_t-c\d_x^2$.
The aforementioned transformation then establishes that \eqref{eq:main} is equivalent, locally in time, to
\begin{equs}[eq:main-t]
u & =I(a(u),\hat f)
\\
\hat f &=\big(1-a'(u)I_c(a(u),\hat f)\big)\xi + (aa'')(u)(\d_x u)^2 I_c (a(u),\hat f)
\\
&\quad+ (a (a')^2)(u)(\d_x u)^2 I_{cc}(a(u),\hat f) + 
2(aa')(u)(\d_x u) I_{cx}(a(u),\hat f),
\end{equs}
where the operators $I_\alpha$, for multiindices $\alpha$ in $c$ and $x$ are defined as
\begin{equ}\label{eq:I}
I_\alpha(b,f)(z)=\int (\d_\alpha P)(b(z),z-z')f(z')\,dz'
\end{equ}
and $I=I_\emptyset$.
Note that $I$ actually extends to $f$ with regularity above $-2$, in which case
\eqref{eq:I} of course needs to be interpreted in the appropriate distributional sense.

One can formulate \eqref{eq:main-t} in the theory of regularity structures as follows.
Start with the regularity structure built as in \cite{H0,BHZ} for the generalised KPZ equation and denote the set of basis vectors (`symbols') by $\cW$,
and the ones with negative degree by $\cW_-$.
Define the `number of integrations' $[\tau]$ recursively by setting
\begin{equ}
\;[X^k]=[\Xi]=0,\quad[\tau\bar\tau]=[\tau]+[\bar\tau],\quad[\cI\tau]=[\cI'\tau]=[\tau]+1.
\end{equ}
Denote $\cB=\cC^{-10}([\lambda,\lambda^{-1}])$ and write $\cB_k$ for the $k$-fold tensor product of $\cB$
with itself, completed under the projective cross norm.
In particular, we have a canonical dense embedding of $\cB_k\otimes \cB_\ell$ into $\cB_{k+\ell}$.
We also use the convention $\cB_0=\R$. 
We then construct
a regularity structure $\scT$ in such a way 
that each symbol $\tau \in\cW$ determines an infinite-dimensional subspace $\cT_\tau$ of the structure 
space $\cT$, isometric to $\cB_{[\tau]}$. To wit, we set
\begin{equ}\label{eq:T}
\cT=\bigoplus_{\alpha} \cT_\alpha\;,\quad
\cT_\alpha:=\bigoplus_{|\tau| = \alpha} T_\tau\;,\quad
\cT_\tau := \cB_{[\tau]}\otimes\span\{\tau\}\;,\quad
\end{equ} 
and equip the spaces $\cT_\alpha$ with their natural norms.
The structure group plays no explicit role for us in this article so we do not address it, the interested reader can find the details in \cite{GH_Q}.
The abstract differentiation, multiplication,
and integration operators on $\cT$ are defined by
\begin{equs}
\scD (\zeta\otimes\tau) & =\zeta\otimes D\tau,
\\
(\zeta\otimes\tau) (\bar \zeta\otimes\bar\tau) & = (\zeta \otimes \bar \zeta)\otimes \tau\bar \tau,
\\
\cI^{\zeta} (\bar \zeta\otimes\tau) & = (\zeta\otimes \bar \zeta)\otimes  \cI\tau.
\end{equs}
Note in particular that the we have a whole family of integration operators $(\cI^\zeta)_{\zeta\in\cB}$.
Note also that the multiplication in general is \emph{not} commutative.
Take a family of kernels $(K^{(c)})_{c\in [\lambda,\lambda^{-1}]}$, which,
along with their derivatives with respect to $c$ up to any finite order,
are uniformly compactly supported and $2$-smoothing in the 
sense of \cite[As~5.1]{H0}.
We will denote $K^\zeta=\zeta(K^{(\cdot)})$ for $\zeta\in\cB$ and $K^{c;\ell}=K^{\d^\ell\delta_c}$.

In the notation of Section \ref{sec:cancel} we set $\hat\cW$ to be the set of all symbols obtained by repeated uses of integration and multiplication. 
Let $\PPi_\eps$ be the canonical model built from $\xi^\eps$ for $\eps>0$, where the dependence on the mollifier $\rho$, which corresponds to $\theta$ in the framework of Section \ref{sec:cancel}, is suppressed.
The fact that Assumption \ref{as:reno} holds
follow from that, due to the spatial symmetry, one has
\begin{equs}
\E & \PPi_\eps\big(\delta_c\otimes\delta_{c'}\otimes\delta_{c''}\otimes\<7>\big)(0)
=\E \PPi_\eps\big(\delta_c\otimes\delta_{c'}\otimes\<4>\big)(0)
=\E\PPi_\eps\big(\delta_c\otimes\delta_{c'}\otimes\<5>\big)(0)
\\
&
=\E\PPi_\eps\big(\delta_c\otimes\<3x>\big)(0)
=\E\PPi_\eps\big(\delta_c\otimes\<3x2>\big)(0)
=\E\PPi_\eps\big(\delta_c\otimes\delta_{c'}\otimes\<2x>\big)(0)
=0
\end{equs}
where $\<0x>$ stands for $X\Xi$.
We then set $\PPi_\eps^{\Sym}:=\PPi_\eps M_{[g_\eps]}$, where $g_\eps$ is from Corollary \ref{cor:main}, and denote the limiting model by $\PPi^{\Sym}=\hat\PPi M_{[g]}$.

We define the maps the $\cK^\zeta$ by replacing $\cI$ and $K$ in
\cite[Eq~5.15]{H0} by $\cI^{\zeta}$ and $K^\zeta$, respectively.
As before, we denote $\cK^{c;\ell}:=\cK^{\partial^\ell\delta_c}$
We can now introduce the lift of the operator $I_\alpha$.
Take two modelled distributions $b$ and $f$ and set $\bar b=\scal{b,\bone}$, $\hat b=b-\bar b$.
If $\d_\alpha=\d_c^k\d_x^m$, then we define
\begin{equ}\label{frI}
\frI_\alpha(b,f)(z):=\sum_{|\ell|\le 10}\frac{(\hat b(z))^{ \ell}}{\ell!}
(\scD^m\cK^{\bar b(z);k+\ell} f)(z)\;.
\end{equ}
It is shown in \cite{GH_Q} that the maps $\frI_\alpha$ satisfy the natural Schauder-estimates on appropriate spaces of modelled distributions.
Assuming for the moment $u_0=0$, the abstract counterpart of \eqref{eq:main-t} then yields the object $u$ claimed in Theorem \ref{thm:main}: We set $u=\cR U$,
where $U$ is the obtained by solving, with respect to the model $\PPi^{\Sym}$, 
the system of abstract equations
\begin{equs}[eq:system]
U &= \frI(a(U), \hat \cF) \;,\\
\hat \cF&= (1- V_c a'(U) )\Xi + 2 V_{cx} a(U)a'(U)\scD U + V_{cc} a(U) (a'(U))^2 (\scD U)^2\\
&\quad + V_c a(U) a''(U) (\scD U)^2
\;,\\
V_\alpha &= \frI_\alpha(a(U), \hat \cF)\;,\quad\quad
\text{for }\alpha=c,cc,cx.
\end{equs}
We will also encounter $V_{ccc}$,
although it does not explicitly appear in the equation \eqref{eq:system}.
For general initial condition $u_0$, one has to include an additional variant of the operators $\frI$
(denoted by $\hat\frI$ in \cite[Eq 4.6]{GH_Q}) in the first and third component of \eqref{eq:system}, but since they do not effect main line of the argument at all, they will be omitted for simplicity.

To prove the theorem, we need to show that if $U^\eps$ is obtained from solving \eqref{eq:system} with respect to $\PPi_\eps^{\Sym}$,
then $u^\eps:=\cR U^\eps$ solves \eqref{eq:main approx}.

The plan is similar to the usual derivation of renormalised equations.
First we use the abstract equation \eqref{eq:system} to derive the form of the expansion of the solution $(U,V_c,V_{cc},V_{cx})$ as well as the `right-hand-side' $\hat\cF$.
That is, for each tree $\tau$ we express the coefficient\footnote{In our terminology the `coefficients' include the distributions attached to the trees. For example in a modelled distribution of the form $H(z)=a_1(z)\zeta_1\otimes\<r1>+a_2(z)\zeta_2(z)\otimes\<r1>$ we say $h_{\<r1>}=a_1\zeta_1+a_2\zeta_2$.}
$u_\tau$ of $\tau$ in $U$, $(v_c)_\tau$ in $V_c$, etc. 
The action of the renormalisation map $M_{[g_\eps]}$ on $\hat \cF$ then produces for each tree $\tau$ a counterterm $g_\eps\big(\hat f_\tau\otimes\tau\big)$ in \eqref{eq:main-t}.
There are two factors complicating this plan.
Firstly, the expansions will have way more terms than one is used to in standard examples like the ones in \cite{H0} - each tree can appear with several different parametrisation on each of its integration edges.
Secondly, the renormalisation of different trees cannot be treated separately: a large number of cancellations have to be exploited to eliminate all nonlocal counterterms and arrive to the (local) ones stated in Theorem \ref{thm:main}. All of these cancellations will come from applications of Lemma \ref{lem:Symmetry}.

\subsection{Notational conventions}\label{sec:notations}
To organise our calculation, let us introduce a couple of shorthand notation.
Firstly, we drop the index $\eps$, but keep it in mind that the solution $(U,V_c,V_{cc},V_{cx})$ we are considering is with respect to the renormalised smooth model $\PPi_\eps^{\Sym}$.
Fix a space-time point $z$ and in the sequel omit the argument $z$ from any function of space-time (scalar-valued and $\cT$-valued alike).
We also omit the $u$ argument from $a$ or any of its derivatives.

In additional to the graphical conventions of Example \ref{example}, we use squares like $\<Z>\,,\<Z2>\,,$
for generic trees, and their color simply serves to distinguish between different ones in the same formula.
Since all symbols appearing in the expansion of the solution
are of the form $\zeta\otimes\tau$, where $\zeta$ is a tensor product of derivatives of $\delta_{a}$, we set the shorthand
$\scal{i_1,\ldots, i_k}=\d^{i_1}\delta_{a}\otimes\cdots\otimes\d^{i_k}\delta_{a}$. 
Furthermore, to ease the reading, we rearrange the order of tensor products.
Given a pictorial representation of a tree,
the ordering is always top-bottom, left-right, but which one takes precedence will change occasionally.
In the notation $\scal{\cdot}$ the order is a) vertical position of the parents b) horizontal position of the parents c) horizontal position of the children (recall that the parent vertex of an edge is the one closer to the root). 
For example,
\begin{equ}
\scal{0,1,2}\otimes\<AMA>=\Xi\big(\cI^{\d^2\delta_{a}}
(\Xi(\cI^{\delta_{a}}\Xi)(\cI^{\d\delta_{a}}\Xi)\big).
\end{equ}
From time to time different ordering of the parameters will be more natural.
In $\scal{\cdot}^\rarr$ we list the parameters in order of a) horizontal position of the parents b) horizontal position of the children c) vertical position of the children.
Finally, in $\scal{\cdot}^\darr$ the order is based on a) vertical position of the parents b) vertical position of the children c) horizontal position of the children.
As examples,
\begin{equ}
\scal{i,j,k,\ell}^\rightarrow\otimes\<S1>=\scal{i,\ell,j,k}\otimes\<S1>\;,
\qquad
\scal{i,j,k,\ell}^\darr\otimes\<S4>=\scal{i,k,j,\ell}\otimes\<S4>.
\end{equ}
We emphasize that these notions all depend on the given pictorial representation.
We further set
\begin{equ}
\sscal{k}_\ell=\sum_{\sum\alpha_i=k}\frac{k!}{\alpha!}\scal{\alpha_1,\ldots,\alpha_\ell}.
\end{equ}
The notation is set up to condense more complicated cancellations.
For example, while Lemma \ref{lem:Symmetry} at first sight only gives
\begin{equ}
g_\eps\big(\scal{0}\otimes\<3>\big) = g_\eps\big( a\sscal{0}_2\otimes\<2>\big),
\end{equ}
one can differentiate this $k$ times with respect to the parameter and obtain
\begin{equ}
g_\eps\big(\scal{k}\otimes\<3>\big) = g_\eps\big(\big(a\sscal{k}_2+k\sscal{k-1}_2\big)\otimes\<2>\big).
\end{equ}
The notations $\sscal{k}^\rarr_\ell,\sscal{k}^\darr_\ell$ are understood analogously.

Recall that $u=\cR U$ and write $v_\alpha=\cR V_\alpha$.
In general, the coefficient of a symbol $\tau\in\cW$ in $U$ will be denoted 
by $u_\tau$, and similarly for $\hat\cF$ and $V_\alpha$.
Some combination of these functions will repeatedly occur: 
\begin{equ}
q=1-v_ca',
\end{equ}
\begin{equ}
p_{c}=\tfrac{1}{q}\big(v_{c}a''+v_{cc}(a')^2\big),\qquad
p_{cc}=\tfrac{1}{q}\big(v_{cc}a''+v_{ccc}(a')^2\big),\qquad
\hat p_c=\tfrac{1}{q}(2 v_{cc}a'a''+v_ca''').
\end{equ}
One important role of $q$ is that, precisely as in \cite{GH_Q}, for short times it is nonzero and $u$ solves an
equation just like \eqref{eq:main}, but with an additional term
\begin{equ}\label{eq:correction}
\tfrac{1}{q}\sum_{\tau\in\cW_-}
g_\eps\big(\hat f_{\tau} \otimes\tau\big).
\end{equ}
appearing on the right-hand side.
Note that while $\cW_-$ is the usual set of negative degree symbols for the generalised KPZ equation,
for each $\tau$, $\hat f_\tau$ is the linear combination of many different distributions, see e.g. \eqref{eq:large} below, and so
\eqref{eq:correction} is in fact a sum of several hundred terms.
Our goal to show that
this sum
is nothing but the counterterm specified in \eqref{eq:main approx},
with the appropriate choice of $C^\eps_\cdot,\bar C^\eps_\cdot,\tilde C^\eps_\cdot$.

To this end, given $\tau\in \cW_-$, for any $k, i_1,\ldots i_{[\tau]}\in\N$, and any function of the form
$h=qa^ka', qa^k(a')^3,qa^ka'a''$, $k\in\N$, we denote $h\scal{i_1,\ldots,i_{[\tau]}}\asymp 0$.
This reflects that the contributions of all terms of the form $h\scal{i_1,\ldots,i_{[\tau]}}\otimes\tau$ to the renormalisation are precisely as required.

We will repeatedly apply integration by parts identities from Section \ref{sec:cancel}.
By writing 
\begin{equ}\label{id n}
\zeta_1\otimes\tau_1\sim\zeta_2\otimes\tau_2			\tag{n}
\end{equ} 
we mean $g_\eps(\zeta_1\otimes\tau_1)=g_\eps(\zeta_2\otimes\tau_2)$.
Given such an identity, we may simplify the expansions of $\hat f_{\tau_1}$ and $\hat f_{\tau_2}$ simultaneously, provided they contain the same multiple of $\zeta_1$ and $\zeta_2$, respectively.
This will be denoted by
\begin{equs}
\hat f_{\tau_1}&= h\zeta_1+\bar\zeta_1
\app{n}\bar\zeta_1,
\\
\hat f_{\tau_2}&= h\zeta_2+\bar\zeta_2\app{n}\bar\zeta_2.
\end{equs}
Here $n$ will be some Roman numeral and $h$ some function.
We emphasise that $\app{n}$ is not a single relation but has to be read in pairs (or, in more complicated situations, triples, quadruples, etc.).
By $\approx$ we mean the summary of all previous simplifications of the coefficient of a given symbol, either by $\asymp$ or $\app{n}$.

It is clear from Gaussianity that only $\tau$-s with $2$ or $4$ instances of $\Xi$ contribute to \eqref{eq:correction}, we denote the corresponding subsets of $\cW_-$ by $\cW_-^2,\cW_-^4$.
With all this, our goal can be summarised as showing $\hat f_\tau\approx0$ for all $\tau\in\cW_-^2,\cW_-^4$.

Finally, let us mention that often the integration by parts will look a bit simpler
due to symbols in $\cW_-^G:=\{\<AMA>\,,\<AMB>\,,\<AMM>\,,\<BMA>\,,\<BMB>\,,\<BMM>\}$
having vanishing contribution to the renormalisation.
This is again a consequence of Gaussianity.
For example, the second to last line in Example \ref{example} simplifies to
\begin{equ}
\sscal{0}_5\otimes\<ABB>\sim a\sscal{0}_6\otimes\<BBB>.
\end{equ}
It is also worth noting and will be often used that these formulae do not require all edges to have the same
parameter. In particular, edges that do not `play' in a given integration by parts, can have arbitrary derivatives, so for example the above relation is true more generally:
\begin{equ}
\scal{0,0,i,j,k}\otimes\<ABB>\sim a\scal{0,0,0,i,k,j}\otimes\<BBB>.
\end{equ}

\begin{remark}\label{rem:Gauss}
One possible way to extend our result to the non-Gaussian case would be to 1) calculate the coefficient $\hat f_\tau$ for $\tau\in\cW^G_-$;
2) keep track of how the performing the steps below effect these coefficients;
3) use the cancellations
relating only elements of $\cW_-^G$ to each other (there are in fact $5$ of these) to further simplify all of these coefficient to $0$ in the sense of $\approx$.
To avoid cluttering the already lengthy computation below, we refrain from this generality.
\end{remark}

\subsection{Some recursions for the coefficients}
First we want to treat the contributions from $\cW^2_-$, but for later use some steps are formulated in a more general way.
In fact, the terms in $\cW^2_-$ had already been treated in \cite{GH_Q}, but for the sake of completeness, as well as to
illustrate the use of some of the notations above, we include the argument.

First of all, it will be repeatedly used that for any $\<Z>\in\cW_-$
one has
\begin{equ}\label{eq:integral with the q}
u_{\<ZI>}=\tfrac{1}{q}\hat f_{\<Z>}\otimes\scal{0}.
\end{equ}
Indeed, this follows from the fact that in $\frI(a(U),\hat\cF)$ the symbol $\<ZI>$ appears twice: once in the $\ell=0$ and once in the $\ell=1$ term. 
Since, by definition, $\cK^{a(u),1}\hat\cF=v_c\bone+(\ldots)$,
one gets the equation
\begin{equ}\label{eq:fixed point coeff}
u_{\<ZI>}=\hat f_{\<Z>}\otimes\scal{0}+a'u_{\<ZI>}v_c ,
\end{equ}
and from it, \eqref{eq:integral with the q}.
One therefore also has
\begin{equs}
(v_{c})_{\<ZI>}&=\hat f_{\<Z>}\otimes\scal{1}+a'u_{\<ZI>}v_{cc}=
\hat f_{\<Z>}\otimes\big(\scal{1}+\tfrac{1}{q}a'v_{cc}\scal{0}\big)\,,
\\
(v_{cc})_{\<ZI>}&=\hat f_{\<Z>}\otimes\scal{2}+a'u_{\<ZI>}v_{ccc}=
\hat f_{\<Z>}\otimes\big(\scal{2}+\tfrac{1}{q}a'v_{ccc}\scal{0}\big)\,,
\\
(v_{cx})_{\<ZJ>}&=\hat f_{\<Z>}\otimes\scal{1}.
\end{equs}
Let $2^*$ denote $2$ for $\<Z>\neq\<Z2>$ and $1$ for $\<Z>=\<Z2>$.
The above then yields the following recursions
\begin{equs}
\hat f_{\<ZA>}
&=-(v_c)_{\<ZI>}a'-v_ca''u_{\<ZI>}
\\
&=-p_c\hat f_{\<Z>}\otimes\scal{0}-a'\hat f_{\<Z>}\otimes\scal{1}\,, \label{eq:rec1}
\\
\hat f_{\<ZZ2>}&=2^*aa'(v_{cx})_{\<ZJ>}\otimes^{\ast}u_{\<Z2I>}+2^*aa'u_{\<ZI>}\otimes^{\ast}(v_{cx})_{\<Z2J>}
+2^*v_{cc}a(a')^2u_{\<ZI>}\otimes^{\ast}u_{\<Z2I>}+2^*v_caa''u_{\<ZI>}\otimes^{\ast}u_{\<Z2I>}
\\
&=\tfrac{1}{q}
\hat f_{\<Z>}\otimes\hat f_{\<Z2>}\otimes
\big(2^*aa'\sscal{1}_2
+2^*ap_c\sscal{0}_2\big)
,\label{eq:rec2}
\end{equs}
where we denoted by $\otimes^\ast$ when the parameter derivatives are slightly rearranged after concatenation (since the way they should be arranged is pretty obvious, we prefer to avoid making this completely precise by introducing further notations).
One obviously has $\hat f_{\<0>}=q$ and we recall the cancellation
\begin{equ}\label{id: i}
\scal{0}\otimes\<3>\sim a\sscal{0}_2\otimes\<2>.\tag{i}
\end{equ}
Thus we can write
\begin{equs}
\hat f_{\<3>}&=-qp_c\scal{0}-qa'\scal{1}
\app{i}-qa'\scal{1}
\asymp 0,
\\
\hat f_{\<2>}
&=qaa'\sscal{1}_2
+qap_c\sscal{0}_2
\app{i}qaa'\sscal{1}_2
\asymp 0.
\end{equs}
The rest of the article is devoted to show $\hat f_{\tau}\approx 0$ for
\begin{equ}
\tau\in\{
\<AAA>\,,\<AAB>\,,\<AAM>\,,\<ABA>\,,\<ABB>\,,\<ABM>\,,
\<BAA>\,,\<BAB>\,,\<BAM>\,,\<BBA>\,,\<BBB>\,,\<BBM>\,,
\<S1>\,,\<S2>\,,\<S3>\,,\<S4>\,,\<S5>\,,
\}.
\end{equ}

The recursions \eqref{eq:rec1}-\eqref{eq:rec2} yield the coefficient of all $8$ of the
above symbols
that are built from the repeated operations $\<Z>\to\<ZA>$, $\<Z>\to\<ZB>$, as well as those of $\<S1>$, $\<S2>$, $\<S3>$.
For the $6$ remaining symbols, however, we have to take into account the fact that $u$ does not only contain symbols of the form $\<ZI>$.
Indeed, one has, by a similar argument as the one leading to \eqref{eq:fixed point coeff},
\begin{equs}[eq:u second order]
u_{\<ZIZ2I>}
&=\tfrac{2^*}{2}\big(
a'u_{\<ZI>}\otimes^\ast\hat f_{\<Z2>}\otimes\scal{1}+a'\hat f_{\<Z>}\otimes u_{\<Z2I>}\otimes^\ast\scal{1}\big)
\\
&\quad
+a'u_{\<ZIZ2I>}v_c
+\tfrac{2^*}{2}\big(a''u_{\<ZI>}\otimes^\ast u_{\<Z2I>}v_c
+(a')^2u_{\<ZI>}\otimes^\ast u_{\<Z2I>}v_{cc}\big)
\\
&
=a'u_{\<ZIZ2I>}v_c+\tfrac{1}{q}\tfrac{2^*}{2}
\hat f_{\<Z>}\otimes\hat f_{\<Z2>}\otimes
\big(a'\sscal{1}_{2}+p_c\sscal{0}_2\big)
\\
&=
\tfrac{1}{q^2}\tfrac{2^*}{2}
\hat f_{\<Z>}\otimes\hat f_{\<Z2>}\otimes
\big(a'\sscal{1}_2+p_c\sscal{0}_2\big)
.
\end{equs}
One also easily gets
\begin{equ}
(v_{cx})_{\<ZIZ2J>}
=a'u_{\<ZI>}\otimes^\ast\hat f_{\<Z2>}\otimes\scal{2}
=\tfrac{1}{q}a'\hat f_{\<Z>}\otimes\hat f_{\<Z2>}\otimes\scal{0,2}.
\end{equ}
From $u_{\<ZIZ2I>}$ we also obtain (here we will only need the case $\<Z>\neq\<Z2>$)
\begin{equs}
(v_{c})_{\<ZIZ2I>}&=a'u_{\<ZI>}\otimes^\ast\hat f_{\<Z2>}\otimes\scal{2}
+a'\hat f_{\<Z>}\otimes u_{\<Z2I>}\otimes^\ast\scal{2}
\\
&\quad
+a'u_{\<ZIZ2I>}v_{cc}
+a''u_{\<ZI>}\otimes^\ast u_{\<Z2I>}v_{cc}+(a')^2u_{\<ZI>}\otimes^\ast u_{\<Z2I>}v_{ccc}
\\
&=
\hat f_{\<Z>}\otimes\hat f_{\<Z2>}\otimes
\tfrac{1}{q}\big[a'\scal{0,2}+a'\scal{2,0}
+\tfrac{1}{q}a'v_{cc}\big(a'\sscal{1}_{2}+p_c\sscal{0}_2\big)
+p_{cc}\sscal{0}_2
\big].
\end{equs}
It turns out that due to the regularity and the Gaussianity of our noise, we will not need to calculate the contributions to $u$ of products with more than $2$ terms.
From now on all product of parameter derivatives will denote a simple concatenation so we drop $\otimes$ from the notation.
The above formulae then yield the three more complicated recursions: for $\<Z>=\<3>\;,\<2>$ (although we only really need $\<Z>\neq\<0>$) we have
\begin{equs}
\hat f_{\<ZAM>}
&=-a''(v_c)_{\<ZI>}u_{\<1>}-a''u_{\<ZI>}(v_c)_{\<1>}-v_ca''u_{\<ZIM>}-a'(v_c)_{\<ZIM>}-v_ca'''u_{\<ZI>}u_{\<1>}
\\
&=-a''\hat f_{\<Z>}			\sscal{1}_{2}
-2\tfrac{1}{q}a'a''v_{cc}\hat f_{\<Z>}			\sscal{0}_2
-p_c\hat f_{\<Z>}			\big(a'\sscal{1}_{2}+p_c\sscal{0}_2\big)
\\
&\quad-\hat f_{\<Z>}			\big((a')^2\scal{0,2}+(a')^2\scal{2,0}+a'p_{cc}\sscal{0}_2\big)
-\tfrac{1}{q}v_ca'''\hat f_{\<Z>}			\sscal{0}_2
\\
&=\hat f_{\<Z>}		\big[-\big(\hat p_c+p_c^2+a'p_{cc}\big)\sscal{0}_2
-\big(p_ca'+a''\big)\sscal{1}_2
-(a')^2\sscal{2}_2
\\
&\quad+ 2(a')^2\scal{1,1}\big].
\end{equs}
Next, we have
\begin{equs}
\hat f_{\<SZ>}&=2aa'(v_{cx})_{\<ZIJM>}u_{\<1>}
+2\big((a')^2+aa''\big)(v_{cx})_{\<r1>}u_{\<ZI>}u_{\<1>}
+2aa'(v_{cx})_{\<r1>}u_{\<ZIM>}
\\
&\quad
+2v_{cc}a(a')^2u_{\<1>}u_{\<ZIM>}
+u_{\<1>}\big((v_{cc})_{\<ZI>}a(a')^2+v_{cc}\big((a')^3+2aa'a''\big)u_{\<ZI>}\big)u_{\<1>}
\\
&\quad
+2v_{c}aa''u_{\<1>}u_{\<ZIM>}
+u_{\<1>}\big((v_{c})_{\<ZI>}aa''+v_{c}\big(a'a''+aa'''\big)u_{\<ZI>}\big)u_{\<1>}
\\
&=
2aa'(v_{cx})_{\<ZIJM>}u_{\<1>}
+2\big((a')^2+aa''\big)(v_{cx})_{\<r1>}u_{\<ZI>}u_{\<1>}
+2aa'(v_{cx})_{\<r1>}u_{\<ZIM>}
\\
&\quad
+2 q a p_c u_{\<1>}u_{\<ZIM>}
+ q a' p_c u_{\<1>} u_{\<ZI>}u_{\<1>}
+ q a \hat p_c u_{\<1>} u_{\<ZI>}u_{\<1>}
\\
&\quad
+a(a')^2  u_{\<1>} (v_{cc})_{\<ZI>}u_{\<1>}
+ aa''     u_{\<1>} (v_{c})_{\<ZI>}u_{\<1>}
\\
&=
\hat f_{\<Z>}			\big[2a(a')^2\scal{0,0,2}
+2 \big((a')^2+aa''\big) \scal{0,0,1}
\\
&\qquad\qquad
+2a(a')^2\scal{1,0,1}
+2a(a')^2\scal{1,1,0}
+2aa'p_c\scal{1,0,0}
\\
&\quad
+2aa'p_c\scal{0}			\sscal{1}_{2}
+2a(p_c)^2\scal{0,0,0}
+a'p_c\scal{0,0,0}
+a\hat p_c\scal{0,0,0}
\\
&\quad
+a(a')^2\scal{0,2,0}
+aa''\scal{0,1,0}
+aa'p_{cc}\scal{0,0,0}\big]
\\
&=
\hat f_{\<Z>}			\big[\big(a\hat p_c+2a(p_c)^2+aa'p_{cc}+a'p_c\big)\sscal{0}_3
+\big(aa''+2aa'p_c\big)\sscal{1}_3
+a(a')^2\sscal{2}_3
\\
&\quad
+2(a')^2\scal{0,0,1}-2a(a')^2\scal{1,1,0}\big].
\end{equs}
Finally, one can write
\begin{equs}
\hat f_{\<ZBM>}&=2aa'(v_{cx})_{\<ZJM>}u_{\<1>}
+2aa'u_{\<ZI>}(v_{cx})_{\<0JM>}
+2\big((a')^2+aa''\big)(v_{cx})_{\<ZJ>}(u_{\<1>})^2
+4aa'(v_{cx})_{\<ZJ>}u_{\<0IM>}
\\
&\qquad\qquad
+2\big((a')^2+aa''\big)u_{\<ZI>}u_{\<1>}(v_{cx})_{\<r1>}
+2aa'u_{\<ZIM>}(v_{cx})_{\<r1>}
\\
&\quad
+2v_{cc}a(a')^2\big(u_{\<ZIM>}u_{\<1>}+2u_{\<ZI>}u_{\<0IM>}\big)
+2u_{\<ZI>}\big((v_{cc})_{\<1>}a(a')^2+v_{cc}\big((a')^3+2aa'a''\big)u_{\<1>}\big)u_{\<1>}
\\
&\quad
+2v_{c}aa''\big(u_{\<ZIM>}u_{\<1>}+2u_{\<ZI>}u_{\<0IM>}\big)
+2u_{\<ZI>}\big((v_{c})_{\<1>}aa''+v_{c}\big(a'a''+aa'''\big)u_{\<1>}\big)u_{\<1>}
\\
&=2aa'(v_{cx})_{\<ZJM>}u_{\<1>}
+2aa'u_{\<ZI>}(v_{cx})_{\<0JM>}
+2\big((a')^2+aa''\big)(v_{cx})_{\<ZJ>}(u_{\<1>})^2
+4aa'(v_{cx})_{\<ZJ>}u_{\<0IM>}
\\
&\qquad\qquad
+2\big((a')^2+aa''\big)u_{\<ZI>}u_{\<1>}(v_{cx})_{\<r1>}
+2aa'u_{\<ZIM>}(v_{cx})_{\<r1>}
\\
&\quad
+2 q a p_c \big(u_{\<ZIM>}u_{\<1>}+u_{\<ZI>}2u_{\<0IM>}\big)
+2 q a' p_c u_{\<ZI>}(u_{\<1>})^2
+2 q a \hat p_c u_{\<ZI>}(u_{\<1>})^2
\\
&\quad
+2a(a')^2u_{\<ZI>}(v_{cc})_{\<1>}u_{\<1>}
+2aa''u_{\<ZI>}(v_{c})_{\<1>}u_{\<1>}		\label{eq:large}
\\
&=
\hat f_{\<Z>}			\big[2a(a')^2\scal{2,0,0}
+2a(a')^2\scal{0,0,2}
+2\big((a')^2+aa''\big)\scal{1,0,0}
\\
&\qquad\qquad
+2 a (a')^2\scal{1,1,0}
+2 a (a')^2\scal{1,0,1}
+2 a a' p_c \scal{1,0,0}
\\
&\qquad\qquad
+2 \big((a')^2+aa''\big) \scal{0,0,1}
\\
&\qquad\qquad
+2a(a')^2\scal{1,0,1}
+2a(a')^2\scal{0,1,1}
+2aa'p_c\scal{0,0,1}
\\
&\quad
+2aa'p_c\big(\sscal{1}_{2}			\scal{0}+\scal{0}			\sscal{1}_{2}\big)
+4a (p_c)^2\scal{0,0,0}
\\
&\qquad\qquad
+2a'p_c\scal{0,0,0}
+2a\hat p_c\scal{0,0,0}
\\
&\quad
+2a(a')^2\scal{0,2,0}
+2aa''\scal{0,1,0}
+2aa'p_{cc}\scal{0,0,0}\big]			
\\
&=
\hat f_{\<Z>}			\big[2\big(a\hat p_c+2a(p_c)^2+aa'p_{cc}+a'p_c\big)\sscal{0}_3
+2\big(aa''+2aa'p_c\big)\sscal{1}_3
+2a(a')^2\sscal{2}_3
\\
&\quad
+2(a')^2\scal{1,0,0}+2(a')^2\scal{0,0,1}
-2a(a')^2\scal{0,1,1}-2a(a')^2\scal{1,1,0}\big].
\end{equs}
\subsection{Exploiting the cancellations}
To simplify the above complicated expressions, a number of application of the identities from Section \ref{sec:cancel} will be needed.
To give some structure to this lengthy computation, in each smaller step we aim to eliminate (in the sense of $\approx$) some terms of a given type.
\\
\emph{Eliminating coefficients with $\hat p_c$, $p_{cc}$, and $a''$}

The coefficients in the above expressions can be viewed as polynomials in the $6$ variables $a,a',a'',p_c,\hat p_c,p_{cc}$, but terms containing three of these can easily be eliminated.
We have the cancellations
\begin{equ}\label{id ii}
\sscal{\ell}_2\otimes\<ZAM>\sim \big(a\sscal{\ell}_3+\ell\sscal{\ell-1}_3\big)\otimes\big(\<SZ>+2\<ZBM>\big).		\tag{ii}
\end{equ}
Applying this with $\ell=0$, and using the notation $\hat f_{\<Z2>}/\hat f_{\<Z>}=\zeta$ to denote $\hat f_{\<Z2>}=\hat f_{\<Z>}\otimes\zeta$, we can write:
\begin{equs}
\hat f_{\<ZAM>}/\hat f_{\<Z>}
&\app{ii}
-p_c^2\sscal{0}_2
-\big(p_ca'+a''\big)\sscal{1}_2
-(a')^2\sscal{2}_2
+ 2(a')^2\scal{1,1},
\\
\hat f_{\<SZ>}/\hat f_{\<Z>}
&\app{ii}
\big(2a(p_c)^2+a'p_c\big)\sscal{0}_3
+\big(aa''+2aa'p_c\big)\sscal{1}_3
+a(a')^2\sscal{2}_3
\\
&\quad
+2(a')^2\scal{0,0,1}-2a(a')^2\scal{1,1,0},
\\
\hat f_{\<ZBM>}/\hat f_{\<Z>}
&\app{ii}
2\big(2a(p_c)^2+a'p_c\big)\sscal{0}_3
+2\big(aa''+2aa'p_c\big)\sscal{1}_3
+2a(a')^2\sscal{2}_3
\\
&\quad
+2(a')^2\scal{1,0,0}+2(a')^2\scal{0,0,1}
-2a(a')^2\scal{0,1,1}-2a(a')^2\scal{1,1,0}.
\end{equs}
Next we apply \eqref{id ii} with $\ell=1$
\begin{equs}
\hat f_{\<ZAM>}/\hat f_{\<Z>}
&\app{ii}
-p_c^2\sscal{0}_2
-p_ca'\sscal{1}_2
-(a')^2\sscal{2}_2
+ 2(a')^2\scal{1,1},
\\
\hat f_{\<SZ>}/\hat f_{\<Z>}
&\app{ii}
\big(2a(p_c)^2+a'p_c-a''\big)\sscal{0}_3
+2aa'p_c\sscal{1}_3
+a(a')^2\sscal{2}_3
\\
&\quad
+2(a')^2\scal{0,0,1}-2a(a')^2\scal{1,1,0},\label{eq: one of many}
\\
\hat f_{\<ZBM>}/\hat f_{\<Z>}
&\app{ii}
2\big(2a(p_c)^2+a'p_c-a''\big)\sscal{0}_3
+4 aa'p_c\sscal{1}_3
+2a(a')^2\sscal{2}_3
\\
&\quad
+2(a')^2\scal{1,0,0}+2(a')^2\scal{0,0,1}
-2a(a')^2\scal{0,1,1}-2a(a')^2\scal{1,1,0}.\label{eq: one of many'}
\end{equs}
We now write
\begin{equs}
\hat f_{\<S4>}
&\approx
q p_c a''\scal{0}			\sscal{0}_3
+ q a' a'' \scal{1}			\sscal{0}_3
+ \rem,
\\
\hat f_{\<S5>}
&\approx
- q a p_c a''\sscal{0}_2			\sscal{0}_3
- q a a' a'' \sscal{1}_2			\sscal{0}_3
+ \rem,
\end{equs}
where $\rem$ stands for all the terms coming from \eqref{eq: one of many} not including $a''$.
Recalling
\begin{equ}\label{id iii}
\scal{0,i,0,j}\otimes\<S4>\sim a\scal{0,0,i,0,j}\otimes\<S5>			\tag{iii}
\end{equ}
for any $i$ and $j$ (although for the moment we only use $i=j=0$), we have
\begin{equs}
\hat f_{\<S4>}
&\app{iii}
q a' a'' \scal{1}			\sscal{0}_3
+ \rem\asymp\rem,
\\
\hat f_{\<S5>}
&\app{iii}
- q a a' a'' \sscal{1}_2			\sscal{0}_3
+ \rem\asymp\rem.
\end{equs}
We therefore have
\begin{equs}
\hat f_{\<SZ>}/\hat f_{\<Z>}
&\approx
\big(2a(p_c)^2+a'p_c\big)\sscal{0}_3
+2aa'p_c\sscal{1}_3
+a(a')^2\sscal{2}_3
\\
&\quad
+2(a')^2\scal{0,0,1}-2a(a')^2\scal{1,1,0},
\end{equs}
and performing the similar steps in \eqref{eq: one of many'}, also
\begin{equs}
\hat f_{\<ZBM>}/\hat f_{\<Z>}
&\approx
2\big(2a(p_c)^2+a'p_c\big)\sscal{0}_3
+4 aa'p_c\sscal{1}_3
+2a(a')^2\sscal{2}_3
\\
&\quad
+2(a')^2\scal{1,0,0}+2(a')^2\scal{0,0,1}
-2a(a')^2\scal{0,1,1}-2a(a')^2\scal{1,1,0}.
\end{equs}
\\
\emph{A remark and eliminating second derivatives}

Note that above argument could of course be easily repeated with $(a')^2$ in place of $a''$.
Therefore, whenever we arrive to
\begin{equ}
\hat f_{\<SZ>}/\hat f_{\<Z>}\approx c a^k(a')^2\scal{i,0,j}+\rem,
\end{equ}
for some $c\in\R$, $i,j,k\in\N$, we can infer
\begin{equ}
\hat f_{\<SZ>}/\hat f_{\<Z>}\approx \rem.
\end{equ}
This simplification will reappear later in the proof, and will be denoted by $\app{S}$.
The analogous statement of course also holds for $\<ZBM>$. Keep in mind that the parameter in the latter case has to be of the form $\scal{0,i,j}$.
We can therefore readily simplify the above to
\begin{equs}
\hat f_{\<SZ>}/\hat f_{\<Z>}
&\app{S}
\big(2a(p_c)^2+a'p_c\big)\sscal{0}_3
+2aa'p_c\sscal{1}_3
+a(a')^2\sscal{2}_3
\\
&\quad
-2a(a')^2\scal{1,1,0},
\\
\hat f_{\<ZBM>}/\hat f_{\<Z>}
&\app{S}
2\big(2a(p_c)^2+a'p_c\big)\sscal{0}_3
+4 aa'p_c\sscal{1}_3
+2a(a')^2\sscal{2}_3
\\
&\quad
+2(a')^2\scal{1,0,0}
-2a(a')^2\scal{1,1,0}.
\end{equs}
To remove the term with $2$ derivatives, simply apply \eqref{id ii} with $\ell=2$:
\begin{equs}
\hat f_{\<ZAM>}/\hat f_{\<Z>}
&\app{ii}
-p_c^2\sscal{0}_2
-p_ca'\sscal{1}_2
+ 2(a')^2\scal{1,1},
\\
\hat f_{\<SZ>}/\hat f_{\<Z>}
&\app{ii}
\big(2a(p_c)^2+a'p_c\big)\sscal{0}_3
+\big(2aa'p_c-2(a')^2\big)\sscal{1}_3
\\
&\quad
-2a(a')^2\scal{1,1,0},
\\
\hat f_{\<ZBM>}/\hat f_{\<Z>}
&\app{ii}
2\big(2a(p_c)^2+a'p_c\big)\sscal{0}_3
+2 \big(2aa'p_c-2(a')^2\big)\sscal{1}_3
\\
&\quad
+2(a')^2\scal{1,0,0}
-2a(a')^2\scal{1,1,0}.
\end{equs}
\\
\emph{Eliminating symbols of the form $\<ZAA>\;,\<ZBA>\;,\<ZAB>\;,\<ZBB>$}

Next we use the identities
\begin{equ}\label{id iv}
\scal{i,\ell}\otimes\<ZAA>+\scal{i,\ell}\otimes\<ZAM>
\sim
2\scal{i}			\big(a\sscal{\ell}_2+2\ell\sscal{\ell-1}_2\big)
\otimes\<ZAB>\;,		\tag{iv}
\end{equ}
Using this with $i=0,1$, $\ell=0,1$, we get
\begin{equs}
\hat f_{\<ZAA>}/\hat f_{\<Z>}
&=
p_c^2\scal{0,0}+a'p_c\scal{0,1}+a'p_c\scal{1,0}+(a')^2\scal{1,1}
\\
&\app{iv} 0,
\\
\hat f_{\<ZAM>}/\hat f_{\<Z>}
&\approx
-p_c^2\sscal{0}_2
-a'p_c\sscal{1}_2
+ 2(a')^2\scal{1,1},
\\
&\app{iv}
-2p_c^2\sscal{0}_2
-2a'p_c\sscal{1}_2
+(a')^2\scal{1,1},
\\
\hat f_{\<ZAB>}/\hat f_{\<Z>}
&=-2ap_c^2\scal{0}			\sscal{0}_2
-2aa'p_c\scal{0}			\sscal{1}_2
-2aa'p_c\scal{1}			\sscal{0}_2
-2a(a')^2\scal{1}			\sscal{1}_2
\\
&\app{iv}
2a'p_c\scal{0}			\sscal{0}_2
+2(a')^2\scal{1}			\sscal{0}_2.		\label{eq: hmk}
\end{equs}
Now we can use \eqref{id ii} again
\begin{equs}
\hat f_{\<ZAM>}/\hat f_{\<Z>}
&\app{ii}
(a')^2\scal{1,1}.
\\
\hat f_{\<SZ>}/\hat f_{\<Z>}
&\app{ii}
-a'p_c\sscal{0}_3
-2(a')^2\sscal{1}_3
-2a(a')^2\scal{1,1,0},
\\
&\app{S}
-a'p_c\sscal{0}_3
-2(a')^2\scal{0,1,0}
-2a(a')^2\scal{1,1,0},
\\
\hat f_{\<ZBM>}/\hat f_{\<Z>}
&\app{ii}
-2a'p_c\sscal{0}_3
-4 (a')^2\sscal{1}_3
+2(a')^2\scal{1,0,0}
-2a(a')^2\scal{1,1,0}
\\
&\app{S}
-2a'p_c\sscal{0}_3
-2 (a')^2\scal{1,0,0}
-2a(a')^2\scal{1,1,0}.
\end{equs}
Similarly to \eqref{id iv}, we have
\begin{equ}\label{id v}
\scal{i,j,\ell}\otimes\<ZBA>+\scal{i,\ell,j}\otimes\<ZBM>\sim
2\scal{i,j}			\big(a\sscal{\ell}_2+\ell\sscal{\ell-1}_2\big)
\otimes\<ZBB>\;,		\tag{v}
\end{equ}
Hence, just as above, we can  write
\begin{equs}
\hat f_{\<ZBA>}/\hat f_{\<Z>}
&=
-2a\big(
p_c^2\sscal{0}_2			\scal{0}
+a'p_c\sscal{1}_2			\scal{0}
+a'p_c\sscal{0}_2			\scal{1}
+(a')^2\sscal{1}_2			\scal{1}
\big)
\\
&\app{v} 0,
\\
\hat f_{\<ZBM>}/\hat f_{\<Z>}
&\approx
-2a'p_c\sscal{0}_3
-2 (a')^2\scal{1,0,0}
-2a(a')^2\scal{1,1,0}
\\
&\app{v}
-2a'p_c\sscal{0}_3
-2 (a')^2\scal{1,0,0}
-2a(a')^2\scal{1,1,0}
\\
&\quad
+2a p_c^2\sscal{0}_3+2aa'p_c\sscal{1}_3+2a(a')^2\scal{1,1,0}+2a(a')^2\scal{0,1,1}
\\
&
\app{S}
\big(2ap_c^2-2a'p_c\big)\sscal{0}_3
+2aa'p_c\sscal{1}_3
-2 (a')^2\scal{1,0,0},
\\
\hat f_{\<ZBB>}/\hat f_{\<Z>}
&=
4a^2\big(
p_c^2\sscal{0}_2			\sscal{0}_2
+a'p_c\sscal{1}_2			\sscal{0}_2
+a'p_c\sscal{0}_2			\sscal{1}_2
+(a')^2\sscal{1}_2			\sscal{1}_2
\big)
\\
&\app{v}
-4aa'p_c\sscal{0}_4
-4a(a')^2\sscal{1}_2			\sscal{0}_2,
\end{equs}
Let us now compare the coefficients of $\<AAB>$ and $\<BAB>\;$.
Using \eqref{eq: hmk} and that  $q(a')^3\scal{1,1}			\sscal{0}_2$, $qa(a')^3\sscal{1}_2			\scal{1}			\sscal{0}_2\asymp 0$,
one can write
\begin{equs}
\hat f_{\<AAB>}&\approx -2qa'
\big(p_c^2\scal{0,0}+a'p_c\scal{1,0}+a'p_c\scal{0,1}\big)			\sscal{0}_2,
\\
\hat f_{\<BAB>}&\approx 2qaa'
\big(
p_c^2\sscal{0}			\scal{0}
+a'p_c\sscal{1}			\scal{0}
+a'p_c\sscal{0}			\scal{1}
\big)			\sscal{0}_2.
\end{equs}
Using 
\begin{equ}\label{id vi}
\sscal{\ell}_2\sscal{0}_2\otimes\<AAB>
\sim
\big
(a\sscal{\ell}_3
+ \ell\sscal{\ell-1}_3
\big)			\sscal{0}_2
\otimes
\<BAB>\tag{vi}
\end{equ}
with $\ell=0,1$, we get
\begin{equs}
\hat f_{\<AAB>}&\app{vi}
0,
\qquad\qquad
\hat f_{\<BAB>}&\app{vi} -2q(a')^2p_c\sscal{0}_5.
\end{equs}
Very similar calculation shows
\begin{equs}
\hat f_{\<ABB>}&\approx
0,
\qquad\qquad
\hat f_{\<BBB>}&\approx 4qa(a')^2p_c\sscal{0}_6.
\end{equs}
Hence, by
\begin{equ}\label{vii}
\sscal{0}_5\otimes\<BAB>
-\sscal{0}_5\otimes \<BBM>
-\sscal{0}_5\otimes\<S5>
\sim 2a\sscal{0}_6\otimes\<BBB>\;,\tag{vii}
\end{equ}
we obtain 
\begin{equ}
\hat f_{\<BAB>}\app{vii} 0,
\qquad\qquad
\hat f_{\<BBB>}\app{vii} 0,
\end{equ}
and we momentarily postpone the effect of \eqref{vii} on $\hat f_{\<BBM>}$, $\hat f_{\<S5>}$.
\\
\emph{Eliminating $\<ABMvar>\;,\<BBMvar>$}

So far the coefficients of the symbols $\<S1>\;,\<S2>\;,\<S3>$ have not at all been simplified.
First, notice that
\begin{equ}
\hat f_{\<S1>}=2qa(a')^3\scal{1,0,1,1}^\rarr+\rem\asymp\rem,
\end{equ}
and similarly for $\<S2>\;,\<S3>$. The terms $\rem$ then only contain parameter derivatives of which at most two is $1$ and the rest is $0$.
From \eqref{eq:rec1}-\eqref{eq:rec2} it is easy to see that these terms are
\begin{equs}
\hat f_{\<S1>}
&\approx qa\sscal{0}_2^\rarr			\big(
p_c^3\sscal{0}_2^\rarr
+2a'p_c^2\sscal{1}_2^\rarr
+2(a')^2p_c\scal{1,1}^\rarr\big)
\\
&\qquad
+2qa(a')^2p_c\sscal{1}_2^\rarr			\scal{0,1}^\rarr
-qa(a')^2p_c\scal{1,0,0,1}^\rarr,
\\
\hat f_{\<S2var>}
&\approx -2qa^2\sscal{0}_3^\rarr			\big(
p_c^3\sscal{0}_2^\rarr
+a'p_c^2\sscal{1}_2^\rarr
+(a')^2p_c\scal{1,1}^\rarr\big)
\\
&\qquad
-2qa^2\sscal{1}_3^\rarr			\big(
a'p_c^2\sscal{0}_2^\rarr
+(a')^2p_c\scal{0,1}^\rarr\big)
\\
&\qquad
-2qa^2(a')^2p_c\sscal{1}_2^\rarr			\sscal{1}_2^\rarr			\scal{0}^\rarr
\end{equs}
Thus, from the cancellations
\begin{equs}[id viii]
\sscal{\ell}_2^\rarr			\scal{i,j}^\rarr\otimes\<S1>
- & \sscal{\ell}_2^\rarr			\scal{i,j}^\rarr\otimes\<ABMvar>\tag{viii}
\\
&\sim \big(a \sscal{\ell}_3^\rarr
+\ell\sscal{\ell-1}_3^\rarr \big)
			\scal{i,j}^\rarr
\otimes\<S2var>			
\end{equs}
we have
\begin{equs}
\hat f_{\<S1>}
&\app{viii}-qa(a')^2p_c\scal{1,0,0,1}^\rarr,
\\
\hat f_{\<S2var>}
&\app{viii}
-qa^2p_c^3\sscal{0}_5^\rarr
+2qa(a')^2p_c\sscal{0}_3^\rarr			
\scal{0,1}^\rarr
\\
&\qquad
-2qa^2\sscal{1}_3^\rarr			
a'p_c^2\sscal{0}_2^\rarr
-2qa^2(a')^2p_c\sscal{1}_2^\rarr			\sscal{1}_2^\rarr			\scal{0}^\rarr,\label{whoa}
\end{equs}
as well as
\begin{equs}[day745]
\hat f_{\<ABMvar>}	&	\app{viii}
-q\big[\big(2ap_c^2-2a'p_c\big)\sscal{0}_3^\rarr
+ 2 a a' p_c \sscal{1}_3^\rarr
- 2 (a')^2 \scal{0,0,1}^\rarr  \big]  			
\big( p_c \scal{0}^\rarr + a' \scal{1}^\rarr\big)
\\
&\qquad
+qa\sscal{0}_2^\rarr			\big(
p_c^3\sscal{0}_2^\rarr
+2a'p_c^2\sscal{1}_2^\rarr
+2(a')^2p_c\scal{1,1}^\rarr\big)
\\
&\qquad
+2qa(a')^2p_c\sscal{1}_2^\rarr			\scal{0,1}^\rarr
\\
&\asymp
\big(-qap_c^3+2qa'p_c^2\big)\sscal{0}_4^\rarr
+2q(a')^2p_c	\sscal{0}_2^\rarr			\sscal{1}_2^\rarr
-2qaa'p_c^2		\sscal{1}_2^\rarr			\sscal{0}_2^\rarr
\end{equs}
where the last step consists of a simple (but somewhat lengthy) rearrangement of terms and 
the fact that $q(a')^3\scal{0,0,1,1}\asymp0$.
One can also rearrange \eqref{whoa} as
\begin{equs}
\hat f_{\<S2>}
&\approx
-qa^2\sscal{0}_2^\rarr			\big(
p_c^3\sscal{0}_3^\rarr
+2a'p_c^2\sscal{1}_3^\rarr
+2(a')^2p_c\scal{1}^\rarr			\sscal{1}_2^\rarr
\big)
\\
&\qquad
-2qa^2(a')^2p_c\sscal{1}_2^\rarr			\scal{0}^\rarr			\sscal{1}_2^\rarr
\\
&\qquad
+2q a (a')^2 p_c \scal{1,0}^\rarr			\sscal{0}_3^\rarr
+2q a^2 (a')^2 p_c \scal{1,0}^\rarr			\scal{0}^\rarr			\sscal{1}_2^\rarr.
\end{equs}
We also have from \eqref{eq:rec2}
\begin{equs}
\hat f_{\<S3>}
&\approx qa^3\sscal{0}_3^\rarr			\big(
p_c^3\sscal{0}_3^\rarr
+2a'p_c^2\sscal{1}_3^\rarr
+2(a')^2p_c\scal{1}^\rarr			\sscal{1}_2^\rarr
\big)
\\
&\qquad
+2qa^3(a')^2p_c\sscal{1}_3^\rarr			\scal{0}^\rarr			\sscal{1}_2^\rarr
-qa^3(a')^2p_c\sscal{1}_2^\rarr			\sscal{0}_2^\rarr			\sscal{1}_2^\rarr.
\end{equs}
Very similar to the above, we have the cancellations
\begin{equs}[id ix]
\sscal{\ell}_2^\rarr			\scal{i,j,k}^\rarr\otimes\<S2>
- & \sscal{\ell}_2^\rarr			\scal{i,j,k}^\rarr\otimes\<BBMvar>\tag{ix}
\\
&\sim \big(a \sscal{\ell}_3^\rarr
+\ell\sscal{\ell-1}_3^\rarr \big)
			\scal{i,j,k}^\rarr
\otimes\<S3>,			
\end{equs}
and so
\begin{equs}
\hat f_{\<S3>}
&\app{ix}
-qa^3(a')^2p_c\sscal{1}_2^\rarr			\sscal{0}_2^\rarr			\sscal{1}_2^\rarr
- 2q a^2 (a')^2 p_c \sscal{0}_4^\rarr			\sscal{1}_2^\rarr
\\
\hat f_{\<S2>}
&\app{ix}
2q a (a')^2 p_c \scal{1}^\rarr			\sscal{0}_4^\rarr
+2q a^2 (a')^2 p_c \scal{1,0}^\rarr			\scal{0}^\rarr			\sscal{1}_2^\rarr,
\end{equs}
and also, similarly to \eqref{day745} but keeping in mind the postponed contribution coming from \eqref{vii} to $\<BBMvar>$,
\begin{equs}[day771]
\hat f_{\<BBMvar>}
&\approx
-2q(a')^2p_c\sscal{0}_5+
\big(qa^2p_c^3-2qaa'p_c^2\big)\sscal{0}_5^\rarr
\\
&\qquad
-2qa(a')^2p_c	\sscal{0}_2^\rarr			\sscal{1}_3^\rarr
+2qa^2a'p_c^2		\sscal{1}_2^\rarr			\sscal{0}_3^\rarr.
\end{equs}
From \eqref{day745}-\eqref{day771} and the identity
\begin{equ}\label{x}
\sscal{i}_2^\rarr			\sscal{\ell}_2^\rarr
\otimes\<ABMvar>
\sim
\sscal{i}_2^\rarr			\big(a\sscal{\ell}_3^\rarr+\ell\sscal{\ell-1}_3^\rarr\big)
\otimes\<BBMvar>		\tag{x}
\end{equ}
we easily conclude
\begin{equ}
\hat f_{\<ABMvar>}\app{x} 0,
\qquad\qquad
\hat f_{\<BBMvar>}\app{x} 0.
\end{equ}
\\
\emph{Finishing up}

Notice next that all remaining terms of $\<S1>\;,\<S2>\;,\<S3>$
have $0$ derivatives on the bottom edges, so integrating by parts there is relatively straightforward.
Using
\begin{equ}\label{id xi}
\scal{i,0,i}^\rarr\otimes\<AAM>\sim a\scal{i,0,0,i}^\rarr\otimes\<S1>\;,		\tag{xi}
\end{equ}
we have, with introducing the shorthand $r=q (a')^2 p_c$
\begin{equs}
\hat f_{\<S1>}
&\app{xi} 0
\\
\hat f_{\<AAM>}
&\approx -q(a')^2\big(p_c\scal{0}^\rarr+a'\scal{1}^\rarr\big)			 \scal{1,1}^\rarr
\asymp -q(a')^2p_c\scal{0,1,1}^\rarr
\\
&\app{xi} -r\sscal{1}_2^\rarr			\scal{1}.
\end{equs}
Similarly we obtain, also recalling the postponed contribution from \eqref{vii} to $\<S5>\,$,
\begin{equs}
\hat f_{\<S3>}	&\approx 0,
\\
\hat f_{\<S5>}  & \approx 
-2q(a')^2p_c\sscal{0}_5
\\
&\qquad
-qa\big(	p_c \sscal{0}_2		+ a' \sscal{1}_2	\big)
			 \big(
a' p_c \sscal{0}_3 + 2 (a')^2 \scal{0,1,0} + 2 a (a')^2  \scal{1,1,0}
\big)
\\
&\qquad
- a^2 r \sscal{1}_2^\darr			\scal{0}^\darr			\sscal{1}_2^\darr
- a r \sscal{0}_2^\darr			\scal{0}^\darr			\sscal{1}_2^\darr
- a r \sscal{1}_2^\darr			\scal{0}^\darr			\sscal{0}_2^\darr
\\
&\asymp
-2r\sscal{0}_5-qaa'p_c^2\sscal{0}_5
\\
&\qquad
-a r  \big(
2 \sscal{1}_2^\darr			\sscal{0}_3^\darr
+2\sscal{0}_2^\darr			\scal{1}^\darr			\sscal{0}_2^\darr
+ \sscal{0}_2^\darr			\scal{0}^\darr			\sscal{1}_2^\darr
\big)
\\
&\qquad
-a^2 r
\sscal{1}_3^\darr			\sscal{1}_2^\darr.
\end{equs}
From the identity
\begin{equ}\label{xii}
\scal{j,k,0,i}^\darr\otimes\<BAM> + \scal{i,0,j,k}^\darr\otimes\<S4>
\sim
2 a \scal{i,0,0,j,k}^\rarr
\otimes\<S2>		\tag{xii}
\end{equ}
we have
\begin{equs}
\hat f_{\<S2>}
&\app{xii}	0
\\
\hat f_{\<BAM>}
&\app{xii}
q a (a')^2 \big( p_c \sscal{0}_2+ a' \sscal{1}_2\big)			\scal{1,1} 
\\
&\qquad
+q (a')^2 p_c \big(
\sscal{0}_2			\scal{0,1}
+ a \sscal{1}_2			\scal{0,1}\big)
\\
&\asymp
r \sscal{0}_3			\scal{1}
+ a r \sscal{1}_3			\scal{1},
\\
\hat f_{\<S4>}
&\app{xii}
q\big(	p_c \scal{0}		+ a' \scal{1}	\big)
			 \big(
a' p_c \sscal{0}_3 + 2 (a')^2 \scal{0,1,0} + 2 a (a')^2  \scal{1,1,0}
\big)
\\
&\qquad
+q (a')^2 p_c \big(
\scal{1,0}^\darr			\sscal{0}_2^\darr
+ a \scal{1,0}^\darr			\sscal{1}_2^\darr\big)
\\
&\asymp
q a' p_c^2 \sscal{0}_4
+r \big( 2 \scal{1,0}^\darr			\sscal{0}_2^\darr
+ 2 \scal{0,1}^\darr			\sscal{0}_2^\darr
\big)
\\
&\qquad
+ a r \sscal{1}_2^\darr			\sscal{1}_2^\darr.
\end{equs}
Now we have
\begin{equ}\label{id: xiii}
\sscal{1}_2			\scal{1}\otimes\<AAM>
\sim
\big(
a \sscal{1}_3
+
\sscal{0}_3
\big)
			\scal{1}
\otimes\<BAM>		\tag{xiii}
\end{equ}
which immediately yields
\begin{equ}
\hat f_{\<AAM>}\app{xiii} 0,
\qquad\qquad
\hat f_{\<BAM>}\app{xiii} 0.
\end{equ}
Finally, let us restate a version of \eqref{id iii} with the ordering $\scal{\cdot}^\darr$:
\begin{equ}\label{id xiv}
\sscal{\ell}_2^\darr			\sscal{i}_2^\darr\otimes\<S4>
\sim 
\big(
a\sscal{\ell}_3^\darr+\ell\sscal{\ell-1}_3^\darr
\big)
			\sscal{i}_2^\darr.
\otimes\<S5>	\tag{xiv}
\end{equ}
Using \eqref{id xiv}
first with $\ell=i=0$, then with $\ell=i=1$, and finally with $\ell=1$, $i=0$:
\begin{equs}
\hat f_{\<S5>}  & \app{xiv}
-2r\sscal{0}_5
\\
&\qquad
-a r  \big(
2 \sscal{1}_2^\darr			\sscal{0}_3^\darr
+2\sscal{0}_2^\darr			\scal{1}^\darr			\sscal{0}_2^\darr
+ \sscal{0}_2^\darr			\scal{0}^\darr			\sscal{1}_2^\darr
\big)
\\
&\qquad
-a^2 r
\sscal{1}_3^\darr			\sscal{1}_2^\darr
\\
&\app{xiv}
-2r\sscal{0}_5-2 ar \sscal{1}_3^\darr			\sscal{0}_2^\darr
\app{xiv}
0,
\\
\hat f_{\<S4>}	& \app{xiv}
r \big( 2 \scal{1,0}^\darr			\sscal{0}_2^\darr
+ 2 \scal{0,1}^\darr			\sscal{0}_2^\darr
\big)
+ a r \sscal{1}_2^\darr			\sscal{1}_2^\darr.
\\
&\app{xiv}
2r\sscal{1}_2^\darr			\sscal{0}_2^\darr
\app{xiv} 0.
\end{equs}
The proof is complete.\qed

\bibliography{quasi}{}

\begin{thebibliography}{OSSW18}
\expandafter\ifx\csname url\endcsname\relax
  \def\url#1{\texttt{#1}}\fi
\expandafter\ifx\csname urlprefix\endcsname\relax\def\urlprefix{URL }\fi
\expandafter\ifx\csname href\endcsname\relax
  \def\href#1#2{#2}\fi
\expandafter\ifx\csname burlalt\endcsname\relax
  \def\burlalt#1#2{\href{#2}{\texttt{#1}}}\fi

\bibitem[BCCH20]{BCCH}
\textsc{Y.~{Bruned}}, \textsc{A.~{Chandra}}, \textsc{I.~{Chevyrev}}, and
  \textsc{M.~{Hairer}}.
\newblock {Renormalizing SPDEs in regularity structures}.
\newblock \emph{To appear in Journal of the European Mathematical Society}
  (2020+).
\newblock \burlalt{arXiv:1711.10239}{http://arxiv.org/abs/1711.10239}.

\bibitem[BDH19]{BDH}
\textsc{I.~Bailleul}, \textsc{A.~Debussche}, and \textsc{M.~Hofmanov{\'a}}.
\newblock {Quasilinear generalized parabolic Anderson model equation}.
\newblock \emph{Stochastics and Partial Differential Equations: Analysis and
  Computations} \textbf{7}, no.~1, (2019), 40--63.
\newblock
  \burlalt{doi:10.1007/s40072-018-0121-1}{http://dx.doi.org/10.1007/s40072-018-0121-1}.

\bibitem[BGHZ19]{BGHZ}
\textsc{Y.~Bruned}, \textsc{F.~Gabriel}, \textsc{M.~Hairer}, and
  \textsc{L.~Zambotti}.
\newblock {Geometric stochastic heat equations}.
\newblock \emph{arXiv e-prints} (2019).
\newblock \burlalt{arXiv:1902.02884}{http://arxiv.org/abs/1902.02884}.

\bibitem[BHZ19]{BHZ}
\textsc{Y.~Bruned}, \textsc{M.~Hairer}, and \textsc{L.~Zambotti}.
\newblock Algebraic renormalisation of regularity structures.
\newblock \emph{Inventiones mathematicae} \textbf{215}, no.~3, (2019),
  1039--1156.
\newblock
  \burlalt{doi:10.1007/s00222-018-0841-x}{http://dx.doi.org/10.1007/s00222-018-0841-x}.

\bibitem[CH16]{CH}
\textsc{A.~{Chandra}} and \textsc{M.~{Hairer}}.
\newblock {An analytic BPHZ theorem for regularity structures}.
\newblock \emph{ArXiv e-prints} (2016).
\newblock \burlalt{arXiv:1612.08138}{http://arxiv.org/abs/1612.08138}.

\bibitem[FG19]{FGub}
\textsc{M.~Furlan} and \textsc{M.~Gubinelli}.
\newblock Paracontrolled quasilinear spdes.
\newblock \emph{Ann. Probab.} \textbf{47}, no.~2, (2019), 1096--1135.
\newblock
  \burlalt{doi:10.1214/18-AOP1280}{http://dx.doi.org/10.1214/18-AOP1280}.

\bibitem[GH19a]{GH_Q}
\textsc{M.~Gerencs{\'e}r} and \textsc{M.~Hairer}.
\newblock {A Solution Theory for Quasilinear Singular SPDEs}.
\newblock \emph{Communications on Pure and Applied Mathematics} \textbf{72},
  no.~9, (2019), 1983 -- 2005.
\newblock \burlalt{doi:10.1002/cpa.21816}{http://dx.doi.org/10.1002/cpa.21816}.

\bibitem[GH19b]{GH17}
\textsc{M.~Gerencs{\'e}r} and \textsc{M.~Hairer}.
\newblock Singular spdes in domains with boundaries.
\newblock \emph{Probability Theory and Related Fields} \textbf{173}, no.~3,
  (2019), 697--758.
\newblock
  \burlalt{doi:10.1007/s00440-018-0841-1}{http://dx.doi.org/10.1007/s00440-018-0841-1}.

\bibitem[GIP15]{GIP}
\textsc{M.~Gubinelli}, \textsc{P.~Imkeller}, and \textsc{N.~Perkowski}.
\newblock {Paracontrolled distributions and singular PDEs}.
\newblock \emph{Forum of Mathematics, Pi} \textbf{3}, (2015), e6.
\newblock
  \burlalt{doi:10.1017/fmp.2015.2}{http://dx.doi.org/10.1017/fmp.2015.2}.

\bibitem[Hai14]{H0}
\textsc{M.~Hairer}.
\newblock A theory of regularity structures.
\newblock \emph{Inventiones mathematicae} \textbf{198}, no.~2, (2014),
  269--504.
\newblock \burlalt{arXiv:1303.5113}{http://arxiv.org/abs/1303.5113}.
\newblock
  \burlalt{doi:10.1007/s00222-014-0505-4}{http://dx.doi.org/10.1007/s00222-014-0505-4}.

\bibitem[{Hai}16]{H_String}
\textsc{M.~{Hairer}}.
\newblock {The motion of a random string}.
\newblock \emph{ArXiv e-prints} (2016).
\newblock \burlalt{arXiv:1605.02192}{http://arxiv.org/abs/1605.02192}.

\bibitem[Hai18]{H_Takagi}
\textsc{M.~Hairer}.
\newblock {Renormalisation of parabolic stochastic PDEs}.
\newblock \emph{Japanese Journal of Mathematics} \textbf{13}, no.~2, (2018),
  187--233.
\newblock
  \burlalt{doi:10.1007/s11537-018-1742-x}{http://dx.doi.org/10.1007/s11537-018-1742-x}.

\bibitem[Kup16]{Kup}
\textsc{A.~Kupiainen}.
\newblock {Renormalization Group and Stochastic PDEs}.
\newblock \emph{Annales Henri Poincar{\'e}} \textbf{17}, no.~3, (2016),
  497--535.
\newblock
  \burlalt{doi:10.1007/s00023-015-0408-y}{http://dx.doi.org/10.1007/s00023-015-0408-y}.

\bibitem[{Lab}18]{Cyril_Ham}
\textsc{C.~{Labb{\'e}}}.
\newblock {The continuous Anderson hamiltonian in $d \leq 3$}.
\newblock \emph{arXiv e-prints} (2018).
\newblock \burlalt{arXiv:1809.03718}{http://arxiv.org/abs/1809.03718}.

\bibitem[ORS]{Otto_Initial}
\textsc{F.~Otto}, \textsc{C.~Raithel}, and \textsc{J.~a. Sauer}.
\newblock {The Initial Value Problem for Singular SPDEs via Rough Paths}.
\newblock \emph{In preparation} .

\bibitem[OSSW]{OWSS2}
\textsc{F.~{Otto}}, \textsc{J.~{Sauer}}, \textsc{S.~{Smith}}, and
  \textsc{H.~{Weber}}.
\newblock \emph{In preparation} .

\bibitem[OSSW18]{OWSS}
\textsc{F.~{Otto}}, \textsc{J.~{Sauer}}, \textsc{S.~{Smith}}, and
  \textsc{H.~{Weber}}.
\newblock {Parabolic equations with rough coefficients and singular forcing}.
\newblock \emph{arXiv e-prints} (2018).
\newblock \burlalt{arXiv:1803.07884}{http://arxiv.org/abs/1803.07884}.

\bibitem[OW15]{OWeb_Div2}
\textsc{F.~{Otto}} and \textsc{H.~{Weber}}.
\newblock {H\"older regularity for a non-linear parabolic equation driven by
  space-time white noise}.
\newblock \emph{arXiv e-prints} (2015).
\newblock \burlalt{arXiv:1505.00809}{http://arxiv.org/abs/1505.00809}.

\bibitem[OW18]{OWeb_Div}
\textsc{F.~Otto} and \textsc{H.~Weber}.
\newblock Quasi-linear spdes in divergence form.
\newblock \emph{Stochastics and Partial Differential Equations: Analysis and
  Computations} (2018).
\newblock
  \burlalt{doi:10.1007/s40072-018-0122-0}{http://dx.doi.org/10.1007/s40072-018-0122-0}.

\bibitem[OW19]{OW}
\textsc{F.~Otto} and \textsc{H.~Weber}.
\newblock Quasilinear spdes via rough paths.
\newblock \emph{Archive for Rational Mechanics and Analysis} \textbf{232},
  no.~2, (2019), 873--950.
\newblock
  \burlalt{doi:10.1007/s00205-018-01335-8}{http://dx.doi.org/10.1007/s00205-018-01335-8}.

\end{thebibliography}
\bibliographystyle{Martin}

\end{document}